\documentclass[11pt,reqno]{amsart}
\usepackage{amssymb}
\usepackage{amsfonts}
\usepackage{mathrsfs}
\usepackage{amsmath}
\usepackage{graphicx}
\usepackage{hyperref}
\usepackage{float}
\usepackage{epstopdf}
\usepackage{color}
\usepackage{bm}
\usepackage{comment}
\usepackage{soul}

\allowdisplaybreaks
\newtheorem{theorem}{Theorem}[section]
\newtheorem{proposition}[theorem]{Proposition}
\newtheorem{lemma}[theorem]{Lemma}
\newtheorem{corollary}[theorem]{Corollary}

\newtheorem{definition}[theorem]{Definition}

\def\diam{\mathrm{diam}}

\def\AG{\mathcal{AG}}
\def\mcD{\mathcal{D}}
\def\mcE{\mathcal{E}}
\def\mcF{\mathcal{F}}

\def\wwAG{\widetilde{\widetilde{\mathcal{AG}}}}

\numberwithin{equation}{section}

\begin{document}
\title[Convergence of energy forms on Sierpinski gaskets with added rotated triangle]{Convergence of energy forms on Sierpinski gaskets with added rotated triangle}

\author{Shiping Cao}
\address{Department of Mathematics, Cornell University, Ithaca 14853, USA}
\email{sc2873@cornell.edu}
\thanks{}

\subjclass[2010]{Primary 28A80; Secondary 31E05}

\date{}

\keywords{resistance metric, $\Gamma$-convergence, Sierpinski gaskets with added rotated triangles, existence, uniqueness}

\begin{abstract}
  We study the convergence of resistance metrics and resistance forms on a converging sequence of spaces. As an application, we study the existence and uniqueness of self-similar Dirichlet forms on Sierpinski gaskets with added rotated triangle. The fractals depend on a parameter in a continuous way. When the parameter is irrational, the fractal is not post critically finite (p.c.f.), and there are infinitely many ways that two cells intersect. In this case, we define the Dirichlet form as a limit in some $\Gamma$-convergence sense of the Dirichlet forms on p.c.f. fractals that approximate it.
\end{abstract}
\maketitle

\section{introduction}
The study of diffusion processes on fractals emerged as an independent research field in the late 80's, with initial interest coming from mathematical physicists working in the theory of disordered media \cite{AO,HBA,RTP}. On self-similar sets, as pioneering works, Kusuoka and Goldstein \cite{G,kus} independently constructed Brownian motions on the Sierpinski gasket. Important properties were later studied by Barlow and Perkins \cite{BP}, where estimates of transition densities and resolvent kernels are provided.  The original idea, which also applies to the Lindstr{\o}m's nested fractals \cite{Lindstrom}, is to define the Brownian motion as the limit of random walks on the approximating graphs. 

Rich extensions have been developed over the past decades. First, with the introduction of the Dirichlet forms (see the book \cite{FOT}), more analytical techniques are available. Kigami \cite{ki1,ki2} introduced the class of the post-critically finite (p.c.f.) self-similar sets, and constructed self-similar Dirichlet forms as limits of discrete energy forms on approximating graphs. Also, see \cite{HK} by Kumagai and Hambly for the transition density estimates on p.c.f. self-similar sets. The concepts of resistance forms and resistance metrics are raised to explain the idea, and the techniques apply to the broad class of fractals with finitely ramified cell structures \cite{T}, which includes finitely ramified graph-directed fractals \cite{CQ1,HN,HMT} and some Julia sets of polynomials \cite{FS,RT}. In a recent work, we also allow an infinite iterated function system (i.f.s.), see \cite{CQ2} for the Brownian motion on the golden ratio Sierpinski gasket. 

A more difficult extension is the Brownian motion on the class of Sierpinski carpets. The diffusion was first defined as a limit of the reflected Brownian motions \cite{BB} on domains in $\mathbb{R}^2$ that approximate the carpets, and was also constructed with the method of the Dirichlet forms by Kusuoka and Zhou \cite{KZ}. The equivalence of the two constructions remained unknown until 2009 \cite{BBKT}, where a uniqueness theorem was proven. Also, see \cite{BB2} for a higher dimensional extension, and \cite{BB1,BB2} for the transition density estimates. 

Although systematic frameworks (see books \cite{B,ki3,s}) on self-similar sets are developed, the techniques are based on the assumption that there are finitely many ways that two small copies of fractals generated by the i.f.s., which we usually call cells, can intersect. For example, in \cite{KZ}, the condition (GB) was raised to deal with the boundary. The class of Sierpinski gaskets with added rotated triangle, which we will study in this paper, were raised in \cite{B} as examples to show that the delicacy of the condition. 

Let $p_1=(\frac{1}{2},\frac{\sqrt{3}}{2})$, $p_2=(0,0)$ and $p_3=(1,0)$, which are the vertices of an equilateral triangle in $\mathbb{R}^2$, and define $F_i(x)=\frac{1}{2}x+\frac{1}{2}p_i$ for $i=1,2,3$. Then, the standard Sierpinski gasket is defined to be the unique compact set $\mathcal{SG}$ such that $\mathcal{SG}=\bigcup_{i=1}^3F_i\mathcal{SG}$. With a fixed parameter $\lambda\in (0,\frac{1}{2})$, the Sierpinski gasket with added rotated triangle $\AG_\lambda$, as the name implies, is the attractor
\[\AG_\lambda=\bigcup_{i=1}^4 F_{i,\lambda}\AG_\lambda,\]
where $F_{i,\lambda}=F_i$ for $i=1,2,3$ and $F_{4,\lambda}$ is the similarity map depending on $\lambda$ such that
\[F_{4,\lambda}(p_1)=(\frac{1}{4}+\lambda,\frac{\sqrt{3}}{4}),\quad F_{4,\lambda}(p_2)=(\frac12-\frac\lambda2,\frac{\sqrt{3}}{2}\lambda),\quad F_{4,\lambda}(p_3)=(\frac34-\frac\lambda2,\frac{\sqrt{3}}{4}-\frac{\sqrt{3}}{2}\lambda).\]
See Figure \ref{fig1} for an example with $\lambda=\frac{1}{\sqrt{8}}$. In particular, $\AG_\lambda$ is p.c.f. when $\lambda\in (0,\frac{1}{2})\cap\mathbb{Q}$, and $\AG_\lambda$ is non p.c.f. with infinitely many types of intersections between two cells (although they always intersect at a single point) when $\lambda \in (0,\frac{1}{2})\setminus \mathbb{Q}$.

\begin{figure}[htp]
	\includegraphics[width=6cm]{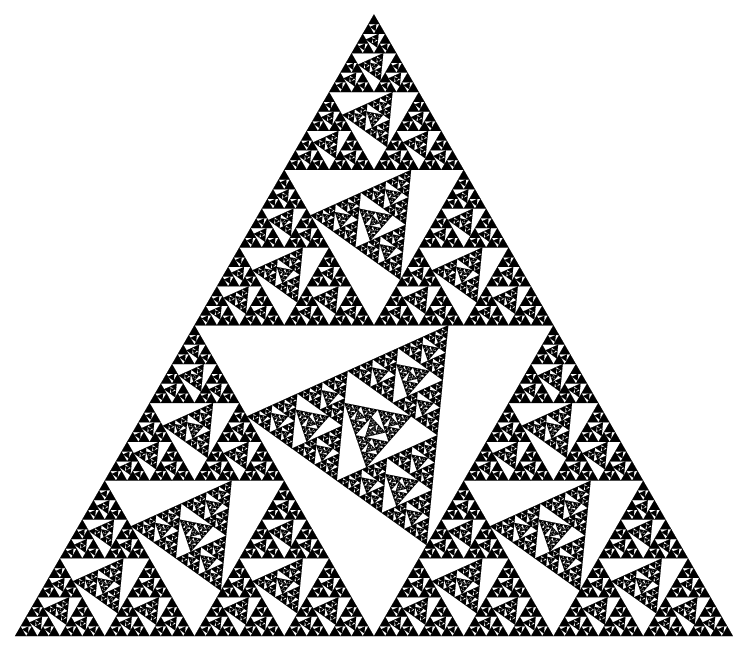}
	\caption{$\AG_{1/\sqrt{8}}$}\label{fig1}
\end{figure}

It is of great interest to study the self-similar Dirichlet forms on $\AG_\lambda,\lambda\in (0,\frac12)\setminus\mathbb{Q}$ as an explorative work to break the condition that cells intersect in finitely many ways. We hope this will lead to further studies on more complicated self-similar sets, which are not finitely ramified.

In this paper, we will answer the following two questions on $\AG_\lambda$.\vspace{0.15cm}

\noindent\textbf{Q1}. \emph{Is there a self-similar resistance form on $\AG_\lambda$ for any $\lambda\in (0,\frac{1}{2})$ (given suitable renormalization factors)? Is the form unique?}\vspace{0.15cm}

\noindent\textbf{Q2}. \emph{How do the forms depend on the parameter $\lambda$ (and the renormalization factors)?}\vspace{0.15cm}

Q1 has been a central problem in the field of analysis on fractals, which has been deeply studied on the p.c.f. self-similar sets \cite{M1,M2,Pe,Sabot} and on the Sierpinski carpets \cite{BBKT}, while Q2 is new in the fractal settings. We provide a positive answer to Q1 with Theorem \ref{thm36}, and also provide an answer to Q2 with Theorem \ref{thm38} and Corollary \ref{coro39}. Readers can refer to \cite{GMS,KS} for some consequences of Corollary \ref{coro39}.

We will begin our story with an investigation of Q2 in the more general setting of resistance spaces. For convenience, we consider a sequence of compact sets $A_n$ and a compact $A$, that are embedded in a same locally compact seperable metric space $(B,d)$ (in fact, compact by choosing a neighbour of $A$), and assume $A_n$ converges to $A$ in the sense of the Hausdorff metric. If there is some uniform estimate on the resistance metrics on $A_n$, we can find a resistance metric on $A$ as a limit by passing to a subsequence (Theorem \ref{thm29}). The corresponding forms converge in a $\Gamma$-convergence sense (Theorem \ref{thm213}), and the resolvent kernels converge (Theorem \ref{thm217}). In addition to our analytic results, we also need to point out here that in \cite{C1,CHK}, deep studies on convergencing resistance spaces are also done in a probability point of view. In particular, the weak convergence (in the Skorohod $J_1$-topology sense) of the associated Hunt processes is proven under mild assumptions (\cite{C1,CHK}).

In history, the convergence of quadratic forms on a converging sequence of spaces was studied in \cite{KS}, based on the theorems of $\Gamma$-convergence \cite{D} and Mosco's work \cite{Mosco}. Readers can find various consequences of the convergence of resolvent kernels in \cite{KS}, including the convergence of the spectrum. Some criteria are rasied in \cite{KS} for the existence of a converging subsequence of forms, but the conditions heavily depend on the energy measures, which are not convenient in the fractal settings, noticing that the energy measures are singular with respect to the self-similar measures \cite{Hino}. 

As an application of the theorems we developed on resistance spaces, we will take a method of approximating fractals with fractals towards Q1. This finishes our story.

At the end of this section, we briefly talk about the structure of this paper. Our paper can be devided into two main parts, Section 2 and Section 3-6. In Section 2, we will have an investigation on the convergence of resistance metrics and resistance forms. Then, in Section 3-6, we solve Q1 and Q2 on $\AG_\lambda,\lambda\in(0,\frac{1}{2})$. In Section 3, we introduce some important notations and the main results. In Section 4, we will provide a positive answer to Q1 when $\lambda$ is a dyadic rational. In Section 5, we will prove some resistance estimates to verify the conditions of Theorem \ref{thm29}. Finally, in Section 6, we apply the theorems in Section 2 to solve Q1 and Q2. Lastly, we end this paper with a short discussion in Section 7 about irregular cases, which means the renormalization factors can be greater than $1$.

Finally, we point out that our techniques apply to some other classes of fractals. For example, the higher dimensional analogs to $\AG_\lambda$, and we may also add more triangles.

\section{Convergence of resistance metric}
The resistance metrics and the resistance forms, introduced by J. Kigami \cite{ki3},  provide a convenient way to generate a Dirichlet form on a finitely ramified fractal. In this section, we assume that we have a sequence of resistance metrics defined on a converging sequence of sets, and try to define a limit resistance metric under some conditions that are verfiable.

We will consider the following settings. The conditions are clearly not optimal, and we hope they can be improved in the future. \vspace{0.15cm}

\noindent\textbf{(C1)}. \textit{$\{A_n\}_{n\geq 1}$ is a sequence of compact sets in a compact metric space $(B,d)$, $A$ is a compact subset of $B$, and 
	\[\lim_{n\to\infty}\delta(A_n,A)=0,\]
where $\delta$ is the Hausdorff metric on compact subsets of $B$.}

\noindent\textbf{(C2)}. \textit{For each $n\geq 1$, $\mu_n$ is a probability measure with $A_n$ being its support; $\mu$ is a probability measure with $A$ being its support. We assume $\mu_n\Rightarrow \mu$.}\vspace{0.15cm}

In (C1), the Hausdorff metric $\delta(A,A')$ between compact subsets $A,A'$ of $B$ is defined as \[\delta(A,A')=\inf\{\rho>0: A\subset B_\rho(A')\text{ and }A'\subset B_\rho (A)\},\]
where $B_\rho(x)=\{y\in B:d(x,y)<\rho\}$ is an open ball of diameter $\rho$, and $B_\rho(A)=\bigcup_{x\in A}B_\rho(x)$ is an open neighbourhood of $A$.

In (C2), $\mu_n\Rightarrow \mu$ means that $\mu_n$ converges weakly to $\mu$, i.e. 
\[\int_B fd\mu_n\to \int_B fd\mu,\quad \forall f\in C(B),\]
noticing that we always assume $B$ is compact throughout the paper. We use the same notations for bounded signed Borel measures.

In particular, if $\lambda_n\to\lambda\in (0,\frac{1}{2})$, then the sequence $\{\AG_{\lambda_n}\}_{n\geq1}$ and $\AG_\lambda$ satisfy condition (C1). See Section 3 for details. If we take the normalized Hausdorff measure on $\AG_\lambda,\lambda\in (0,\frac{1}{2})$, (C2) is satisfied. 

From Section 3 to Section 6, we will study $\AG_\lambda,\lambda\in (0,\frac{1}{2})$ as an application of the results in this section. 

\subsection{Lemmas about convergence} First, we need to understand how a sequence of functions $f_n\in C(A_n),n\geq 1$ converges to a function $f\in C(A)$. 
 
\begin{definition}\label{def21}
	Assume (C1), let $f_n\in C(A_n)$ for $n\geq 1$ and let $f\in l(A)$. For any topological space $X$, $\|\bullet\|_{C(X)}$ denotes the supremum norm on $C(X)$.
	
	(a). We write $f_n\rightarrowtail f$ if $f(x)=\lim\limits_{n\to\infty}f_n(x_n)$ for any $x\in A$ and $x_n\in A_n,n\geq 1$ such that $x_n\to x$ as $n\to\infty$. 
	
	(b). We say $\{f_n\}_{n\geq 1}$ is uniformly bounded if $\sup_{n\geq 1}\|f_n\|_{C(A_n)}<\infty$.
	
	(c). We say $\{f_n\}_{n\geq 1}$ is equicontinuous if 
	\[\lim_{\delta\to 0}\sup\{|f_n(x)-f_n(y)|:x,y\in A_n,d(x,y)<\delta,n\geq 1\}=0.\]
	
\noindent The concepts apply to a subsequence $\{f_{n_k}\}_{k\geq 1}$ in a natural way. 
\end{definition}

\begin{lemma}\label{lemma22}
	Assume (C1) and let $f_n\in C(A_n)$ for $n\geq 1$.
	
(a). If $\{f_n\}_{n\geq 1}$ is uniformly bounded and equicontinuous, then there is a subsequence $\{f_{n_k}\}_{k\geq 1}$ and $f\in C(A)$ such that $f_{n_k}\rightarrowtail f$.

(b). If $f_n\rightarrowtail f$ for some $f\in l(A)$, then $\{f_n\}_{n\geq 1}$ is uniformly bounded and equicontinuous, and $f\in C(A)$.  
\end{lemma}
\begin{proof}
(a). The lemma is an analogue to the Arzel\`{a}-Ascoli theorem. Let $\{x_m\}_{m\geq 1}$ be a countable dense subset of $A$. For each $m\geq 1$, choose a sequence $\{x_{m,n}\}_{n\geq 1}$ so that $x_{m,n}\in A_n,\forall n\geq1$ and $d(x_{m,n},x_m)\to 0$ as $n\to\infty$. Since $\{f_n\}_{n\geq 1}$ is uniformly bounded, by a diagonalization argument, there is a subsequence $\{n_k\}_{k\geq 1}$ so that the limit $\lim_{k\to\infty}f_{n_k}(x_{m,n_k})$ exists for any $m\geq 1$, and we denote $f(x_m)=\lim_{k\to\infty}f_{n_k}(x_{m,n_k})$. Clearly, for $m,m'\geq 1$, 
	\[
	|f(x_{m'})-f(x_m)|\leq \limsup_{n\to\infty}|f_n(x_{m',n})-f_n(x_{m,n})|,\quad d(x_{m'},x_m)=\lim\limits_{n\to\infty}d(x_{m',n},x_{m,n})
	\] 
	so $f$ is uniformly continuous on $\{x_m\}_{m\geq 1}$ as $\{f_n\}_{n\geq 1}$ is equicontinuous. Thus, $f$ extends to a continuous function on $A$. Finally, it is routine to check $f_{n_k}\rightarrowtail f$. For any $\varepsilon>0$ and $y_k\in A_{n_k},k\geq 1$ such that $y_k\to y$, by the equicontinuity of $\{f_n\}_{n\geq 1}\cup \{f\}$, we can choose $x_m$ close enough to $y$ and choose $N$ large enough so that $|f_{n_k}(y_k)-f(y)|\leq |f_{n_k}(y_k)-f_{n_k}(x_{m,n_k})|+|f_{n_k}(x_{m,n_k})-f(x_m)|+|f(x_m)-f(y)|<\varepsilon$ for any $k\geq N$. 

(b) is an analogue of the well-known result. If $\{f_n\}_{n\geq 1}$ is not uniformly bounded, then we can find a sequence $x_{n_k}\in A_{n_k},k\geq 1$ so that $|f_{n_k}(x_{n_k})|\to +\infty$ as $k\to\infty$. Clearly, by passing to a subsequence, we have $x_{n_k}$ converges in $B$ to $x$. Since $\delta(A_{n_k},A)\to 0$, we have $x\in \bigcap_{\delta>0}B_\delta(A)=A$. This contradicts $f_n\rightarrowtail f$. Similarly, we can show $\{f_n\}_{n\geq 1}$ is equicontinuous. Finally, $f$ is continuous by (a). 
\end{proof}

The following lemma helps clarify the meaning of $f_n\rightarrowtail f$. For short, we write $g_n\rightrightarrows g$ if $\{g_n\}_{n\geq 1}\cup\{g\}\subset C(B)$ and $\lim\limits_{n\to\infty}\|g_n-g\|_{C(B)}=0$ as $n\to\infty$. 

\begin{proposition}\label{prop23}
	Assume (C1), let $f_n\in C(A_n)$ for $n\geq 1$ and $f\in C(A)$. The following (a), (b) and (c) are equivalent. 
	
	(a). $f_n\rightarrowtail f$.
	
	(b). Let $\tilde{A}=(\{0\}\times A)\bigcup\big(\bigcup_{n\geq 1}(\{\frac{1}{n}\}\times A_n)\big)$, with the induced topology from $[0,1]\times B$. Define $\tilde{f}\in l(\tilde{A})$ by 
	\[\tilde{f}(t,x)=\begin{cases}
	f_n(x),&\text{ if }t=\frac{1}{n},\\
	f(x),&\text{ if }t=0.
	\end{cases}\]
	We have $\tilde{f}\in C(\tilde{A})$.
	
	(c). There exists $\{g_n\}_{n\geq 1}\cup \{g\}\subset C(B)$ such that $g_n|_{A_n}=f_n,\forall n\geq 1$, $g|_A=f$ and $g_n\rightrightarrows g$. 
\end{proposition}
\begin{proof}
	`(a)$\Rightarrow$(b)' and `(c)$\Rightarrow$(a)' are trivial. 
	
	`(b)$\Rightarrow$(c)'. It is easy to see that  $\tilde{A}$ is a closed subset of $[0,1]\times B$. By the Tietze extension theorem, we have a continuous extension $\bar{f}\in C([0,1]\times B)$ of $\tilde{f}$. It suffices to take $g_n(x)=\bar{f}(\frac{1}{n},x)$ and $g(x)=\bar{f}(0,x)$ for any $x\in B,n\geq 1$. 
\end{proof}	

A similar proof gives us the following lemma.
\begin{lemma}\label{lemma24}
	Assume (C1), let $f_n\in C(A_n)$ for $n\geq 1$ and $g\in C(B)$. Suppose that $f_n\rightarrowtail g|_A$, then there are $\{g_n\}_{n\geq 1}\subset C(B)$ such that $g_n|_{A_n}=f_n$ for $n\geq 1$ and $g_n\rightrightarrows g$. 
\end{lemma}

Lastly, we need to consider the measure. As an immediate consequence of Proposition \ref{prop23} `(a)$\Leftrightarrow$(c)', we have the following observation.

\begin{lemma}\label{lemma25}
	Assume (C1), (C2). Let $f_n\in C(A_n)$ for $n\geq 1$ and let $f\in C(A)$. If $f_n\rightarrowtail f$, then we have $f_n\mu_n\Rightarrow f\mu$.  
\end{lemma}

\subsection{Convergence of resistance metric} On a finite set $V$, a quadratic form $(\mcE,l(V))$ is called a \textit{resistance form} on $V$ if it takes the form
\[\mcE(f)=\frac{1}{2}\sum_{x\neq y}c_{x,y}\big(f(x)-f(y)\big)^2,\quad\forall f\in l(V)\]
where $c_{x,y}=c_{y,x}\in [0,\infty)$ is called the conductance between $x,y$, and in addition $\mcE(f)=0$ if and only if $f\in Constants$, where we denote the space of constant functions by $Constants$. Since the domain is always $l(V)$, we simply refer to $\mathcal{E}$ for the form from time to time.

The effecitive resistance $R$ associated with $\mcE$ is defined by $R(x,y)=\sup\{\frac{|f(x)-f(y)|^2}{\mcE_V(f)}:f\in l(V)\setminus constants\}$. It is well-known that $R$ is a metric on $V$. 

\begin{definition}[\cite{ki3}]\label{def26}
	Let $X$ be a set. $R\in l(X\times X)$ is called a resistance metric on $X$ if for any finite subset $V$ of $X$, there exists a resistance form $\mcE_V$ on $V$ such that 
	\[R(x,y)=\sup\{\frac{|f(x)-f(y)|^2}{\mcE_V(f)}:f\in l(V)\setminus Constants\}.\]
	Let $\mathcal{RM}(X)$ denote the set of resistance metrics on $X$. 
\end{definition}

In this subsection, we will see that given some good resistance metrics on $A_n$, we can find a limit resistance metric on $A$. 

\begin{lemma}\label{lemma27}
	Let $(\mcE_1,l(V))$ and $(\mcE_2,l(V))$ be two resistance forms on a finite set $V$ with $\#V=N\geq 2$. For $i=1,2$, let $R_i$ be the resistance metric associated with $\mcE_i$, i.e. 
	\[R_i(x,y)=\sup\{\frac{|f(x)-f(y)|^2}{\mcE_i(f)}:f\in l(V)\setminus Constants\}.\]  
	Then, we have 
	\[\frac{2}{N(N-1)}\min_{x\neq y}\frac{R_1(x,y)}{R_2(x,y)}\mcE_1(f)\leq \mcE_2(f)\leq \frac{N(N-1)}{2}\max_{x\neq y}\frac{R_1(x,y)}{R_2(x,y)}\mcE_1(f),\quad\forall f\in l(V).\]
\end{lemma}
\begin{proof}
	For $i=1,2$, let $c_{x,y,i}$ denote the conductances between $x,y$ associated with $\mcE_i$, i.e. $\mcE_i$ have the form $\mcE_i(f)=\frac{1}{2}\sum_{x\neq y}c_{x,y,i}\big(f(x)-f(y)\big)^2$. It is clear that $c_{x,y,i}\leq R_i(x,y)^{-1},\forall x\neq y,i=1,2.$ So for any $f\in l(V)$, we have
	\[
	\max_{x\neq y}R_i(x,y)^{-1}\big(f(x)-f(y)\big)^2\leq\mcE_i(f)\leq\frac{1}{2}\sum_{x\neq y}R_i(x,y)^{-1}\big(f(x)-f(y)\big)^2.
	\]
	Thus, for any $f\in l(V)\setminus Constants$,
	\[
	\begin{aligned}
	\frac{\mcE_2(f)}{\mcE_1(f)}&\leq \frac{1}{2}\frac{\sum_{x\neq y}R_2(x,y)^{-1}\big(f(x)-f(y)\big)^2}{\max_{x\neq y}R_1(x,y)^{-1}\big(f(x)-f(y)\big)^2}\\
	&\leq \frac{N(N-1)}{2}\frac{\sum_{x\neq y}R_2(x,y)^{-1}\big(f(x)-f(y)\big)^2}{\sum_{x\neq y}R_1(x,y)^{-1}\big(f(x)-f(y)\big)^2}\leq \frac{N(N-1)}{2}\max_{x\neq y}\frac{R_1(x,y)}{R_2(x,y)}.
	\end{aligned}
	\]
	The other direction follows from a same argument.  
\end{proof}

In the following, $\mathcal{RF}(V)$ denotes the set of resistance forms on a finite set $V$, and $\mathcal{M}(X)$ denotes the set of metrics on a set $X$.

\begin{lemma}\label{lemma28}
	Let $V$ be a finite set. We assign the topology on $\mathcal{RM}(V),\mathcal{M}(V)$ and $\mathcal{RF}(V)$ by naturally embedding them into $\mathbb{R}^{\frac{N(N-1)}{2}}$, where $N=\#V$.
	
	(a). $\mathcal{RM}(V)$ is a closed subset of $\mathcal{M}(V)$.
	
	(b). The natural correspondence $\mathcal{RM}(V)\to \mathcal{RF}(V)$ is a homeomorhpism. 
\end{lemma}
\begin{proof}
	Let $\{R_n\}_{n\geq1}\subset\mathcal{RM}(V)$, and let $\mcE_n$ be the resistance form associated with $R_n$. Assume that there is $R\in \mathcal{M}(V)$ such that $\lim_{n\to\infty}R_n(x,y)=R(x,y),\forall x,y\in V$. We need to show that $R\in \mathcal{RM}(V)$ and $\mcE_n\to \mcE$, where $\mcE$ is the resistance form associated with $R$. 
	
	For $n\geq 1$, we write $\mcE_n(f)=\frac{1}{2}\sum_{x\neq y}c_{x,y,n}\big(f(x)-f(y)\big)^2,\forall f\in l(V)$. By choosing a subsequence, we have the limit $c_{x,y}=\lim_{n_k\to\infty}c_{x,y,n_k}$ exists for any $x\neq y\in V$, and we define $\mcE$ by 
	\[\mcE(f)=\frac{1}{2}\sum_{x\neq y}c_{x,y}\big(f(x)-f(y)\big)^2,\quad\forall f\in l(V).\]
	By Lemma \ref{lemma27}, there exists $1\leq c<\infty$ such that 
	\[c^{-1}\mcE_1(f)\leq \mcE_n(f)\leq c\mcE_1(f),\quad\forall f\in l(V),\forall n\geq 1. \]
	As a consequence, $\mcE(f)>0$ for any $f\in l(V)\setminus Constants$. Thus $\mcE\in \mathcal{RF}(V)$ and $\mcE_{n_k}\to \mcE$. It is easy to see that the correspondence $\mathcal{RF}(V)\to\mathcal{RM}(V)$ is continuous, so $R$ is the resistance metric associated with $\mcE$. For each subsequence, we can do the same construction, and the limit $\mcE$ is the same since it is uniquely determined by $R$. So $\mcE_n\to\mcE$ as $n\to\infty$.
\end{proof}

Returning to the basic setting (C1), the following theorem provides a condition that we can find a converging subsequence of resistance metrics.

\begin{theorem}\label{thm29}
	Assume (C1). Let $R_n\in \mathcal{RM}(A_n)$ for $n\geq 1$, and assume  
	\[\psi_1\big(d(x,y)\big)\leq R_n(x,y)\leq \psi_2\big(d(x,y)\big),\quad\forall n\geq 1,\forall x,y\in A_n \]
	for some $\psi_1,\psi_2\in C[0,\infty)$ with $\psi_1(0)=\psi_2(0)=0$ and $\psi_2(t)\geq \psi_1(t)>0$ for $t>0$. 
	
	Then, there exists $R\in \mathcal{RM}(A)$ and a subsequence $\{n_k\}_{k\geq 1}$ such that $R_{n_k}\rightarrowtail R$. Clearly,
	\[\psi_1\big(d(x,y)\big)\leq R(x,y)\leq \psi_2\big(d(x,y)\big),\quad\forall x,y\in A.\]
\end{theorem}
\begin{proof}
	By Lemma \ref{lemma22}, there exist a subsequence $\{n_k\}_{k\geq 1}$ and $R\in C(A\times A)$ so that $R_{n_k}\rightarrowtail R$. As $d(x_n,y_n)\to d(x,y)$ for any $x_n\to x,y_n\to y$, we have 
	\[R(x,y)=\lim_{k\to\infty}R_{n_k}(x_{n_k},y_{n_k})\geq \lim_{k\to\infty}\psi_1\big(d(x_{n_k},y_{n_k})\big)=\psi_1\big(d(x,y)\big).\]
	So $R\in \mathcal{M}(A)$. The upper bound $R(x,y)\leq \psi_2\big(d(x,y)\big)$ follows from a similar argument.
	
	It remains to show that $R\in \mathcal{RM}(A)$. Without loss of generality, we assume that $R_n\rightarrowtail R$. Let $V=\{x_m\}_{m=1}^M$ be a finite subset of $A$. For each $1\leq m\leq M$, we can choose a sequence $\{x_{m,n}\}_{n\geq1}$ with $x_{m,n}\in A_n,\forall n\geq 1$, so that $x_{m,n}\to x_m$ as $n\to\infty$. Clearly, for each $n\geq 1$, $R^{(n)}$, which is defined by $R^{(n)}(x_m,x_{m'})=R_n(x_{m,n},x_{m',n}),\forall 1\leq\forall m,m'\leq M$, is a resistance metric on $V$. Since $R(x_m,x_{m'})=\lim\limits_{n\to\infty}R^{(n)}(x_m,x_{m'})$, $R\in \mathcal{RM}(A)$ by Lemma \ref{lemma28} (a).
\end{proof}

\subsection{$\Gamma-$convergence of resistance forms} In the previous subsection, we see that in some cases, the resistance metrics converge in the sense $R_n\rightarrowtail R$. To make the observation useful, we study some consequences of such convergence. In this subsection, we focus on the resistance forms. 

In this paper, for any $f,g$ in $l(X)$, $f\vee g$ is the function defined as $f\vee g(x)=\max\{f(x),g(x)\}$,  and $f\wedge g$ is the function defined as $f\wedge g(x)=\min \{f(x),g(x)\}$.

\begin{definition}[\cite{ki3}]\label{def210}
	Let $X$ be a set, and $l(X)$ be the space of all real-valued functions on $X$. A pair $(\mathcal{E},\mathcal F)$ is called a \emph{resistance form} on $X$ if it satisfies the following  conditions:
	
	\noindent (RF1). $\mathcal{F}$ is a linear subspace of $l(X)$ containing constants and $\mathcal{E}$ is a nonnegative symmetric quadratic form on $\mathcal F$; $\mathcal{E}(f):=\mathcal{E}(f,f)=0$ if and only if $f$ is constant on X.
	
	\noindent (RF2). Let `$\sim$' be an equivalent relation on $\mathcal{F}$ defined by $f\sim g$ if and only if $f-g$ is constant on X. Then $(\mathcal{F}/\sim, \mathcal{E})$ is a Hilbert space.
	
	\noindent (RF3). For any finite subset $V\subset X$ and  any $u\in l(V)$, there exists a function $f\in \mathcal{F}$ such that $f|_V=u$.
	
	\noindent (RF4). For any distinct $p,q\in X$, $R(p,q):=\sup\{\frac{|f(p)-f(q)|^2}{\mathcal{E}(f)}:f\in \mathcal{F}\setminus Constants\}$ is finite.
	
	\noindent (RF5). If $f\in \mathcal{F}$, then $\bar{f}=(f\vee 0)\wedge 1\in \mathcal{F}$ and $\mathcal{E}(\bar{f})\leq \mathcal{E}(f)$.
	
Following \cite{ki3}, we write $\mathcal{RF}(X)$ for the set of resistance forms on $X$.
\end{definition}

Given a resistance form, (RF4) defines a corresponding resistance metric on $X$; for the reverse direction, if we have a resistance metric $R$ on $X$ so that $(X,R)$ is separable, then we have a unique resistance form $(\mcE,\mcF)$ on $X$ so that $R$ is the associated resistance metric. See \cite{ki3} for details.

\begin{definition}\label{def211}
Assume (C1). Let $(\mcE_n,\mcF_n)$ be quadratic forms defined on $C(A_n)$, and $(\mcE,\mcF)$ be a quadratic form defined on $C(A)$. We say $(\mcE_n,\mcF_n)$ $\Gamma$-converges to $(\mcE,\mcF)$ on $C(B)$ if and only if (a),(b) hold:

(a). If $f_n\rightrightarrows f$, where $\{f_n\}_{n\geq 1}\cup\{f\}\subset C(B)$, then
\[\mcE(f|_A)\leq \liminf_{n\to\infty} \mcE_n(f_n|_{A_n}).\]

(b). For any $f\in C(B)$, there exists a sequence $\{f_n\}_{n\geq 1}\subset C(B)$ such that $f_n\rightrightarrows f$ and 
\[\mcE(f|_A)=\lim_{n\to\infty} \mcE_n(f_n|_{A_n}). \] 
\end{definition}
 
\noindent\textbf{Remark 1.} In Definition \ref{def211} and in the remaining of this paper, for a nonnegative quadratic form $(\mcE,\mcF)$ on a Banach space $H$, we always set $\mcE(f)=+\infty$, if $f\notin \mcF$. So $\mcE:H\to \overline{\mathbb{R}}_+$, and $\mcF=\{f\in H:\mcE(f)<+\infty\}$.

\noindent\textbf{Remark 2.} Let $(X,d)$ be a metric space, let $Y$ be a subset of $X$, and let $(\mcE,\mcF)$ be a form defined on $C(Y)$. For convenience, we will simply write $\mcE(f)=\mcE(f|_A)$ for $f\in C(X)$ if no confusion is caused.

\noindent\textbf{Remark 3.} By Proposition \ref{prop23} and Lemma \ref{lemma24}, we can restate (a),(b) as follows:

(a'). If $f_n\rightarrowtail f$, where $f_n\in C(A_n),\forall n\geq 1$ and $f\in C(A)$, then
\[\mcE(f)\leq \liminf_{n\to\infty} \mcE_n(f_n).\]

(b'). For any $f\in C(A)$, there exists a sequence $\{f_n\}_{n\geq 1}$ such that $f_n\in C(A_n),\forall n\geq 1$,$f_n\rightarrowtail f$ and 
\[\mcE(f)=\lim_{n\to\infty} \mcE_n(f_n). \]

\begin{definition}\label{def212}
	Let $(\mcE,\mcF)\in \mathcal{RF}(X)$, and let $V$ be a finite subset of $X$. We define the \emph{trace} $[\mcE]_V$ of $(\mcE,\mcF)$ onto $V$ as follows,
	\[[\mcE]_V(f)=\min\{\mcE(g):g\in \mcF,g|_{V}=f\},\qquad \forall f\in l(V).\]
	Then $\big([\mcE]_V,l(V)\big)$ is a resistance form on $V$. 
	
	For any $f\in l(V)$, the unique function $g\in \mcF$ such that $\mcE(g)=[\mcE]_V$ is called the \emph{harmonic extension} of $f$ (with respect to $\mcE$ or $(\mcE,\mcF)$).
\end{definition}

\noindent\textbf{Remark}. Let $R$ be the resistance metric associated with $(\mcE,\mcF)$. Then, $[\mcE]_V$ is the unique resistance form on $V$ such that $R(x,y)=\sup\{\frac{|f(x)-f(y)|^2}{[\mcE]_V(f)}:f\in l(V)\setminus constants\},\forall x,y\in V$.\vspace{0.15cm}

In the following, we always use the notations $C(B),C(A_n),C(A),C(A_n^2),C(A^2)$ for spaces of continuous functions with respect to the metric $d$. In particular, a metric $R$ induces the same topology on $B$ (or $A_n,A$) if it is continuous, noticing that $(B,d)$ (or $(A_n,d),(A,d)$) is compact.

\begin{theorem}\label{thm213}
Assume (C1). Let $R_n\in\mathcal{RM}(A_n)\cap C(A_n^2)$ for $n\geq 1$ and $R\in \mathcal{M}(A)$. Assume that $R_n\rightarrowtail R$, so we also have $R\in \mathcal{RM}(A)\cap C(A^2)$. Let $(\mcE_n,\mcF_n)$ be the resistance form associated with $R_n$ for $n\geq 1$, and let $(\mcE,\mcF)$ be the resistance form associated with $R$. Then, we have $(\mcE_n,\mcF_n)$ $\Gamma$-converges to $(\mcE,\mcF)$ on $C(B)$. 
\end{theorem}
\begin{proof}
$R\in \mathcal{RM}(A)\cap C(A^2)$ by Lemma \ref{lemma22} (b) and a same proof of Theorem \ref{thm29}. It suffices to prove (a'),(b') by Remark 3 below Definition \ref{def211}. 

(a'). $R$ induces the same topology as $d$ on $A$, so $(A,R)$ is separable. Let $V_1\subset V_2\subset V_3\subset\cdots$ be a nested sequence of finite subsets of $A$, and assume that $V_*=\bigcup_{m\geq 1}V_m$ is dense in $A$. Then, by the Theorem 2.3.7 in \cite{ki3}, we have 
\[\mcE(f)=\sup_{m\geq 1}[\mcE]_{V_m}(f)=\lim\limits_{m\to\infty} [\mcE]_{V_m}(f),\quad\forall f\in C(A).\]

For each $n\geq 1$, we can find a nested sequence of finite set $V_{1,n}\subset V_{2,n}\subset \cdots$ such that $\delta(V_m,V_{m,n})\to 0$ as $n\to\infty$. In addition, for each $m\geq1$, we require that $\#V_m=\#V_{m,n}$ for $n$ large enough. To achieve this, we choose a converging sequence $x_n\to x$ with $x_n\in A_n$ for each $x\in V_*$. Then, by applying Lemma \ref{lemma28} (b) and the remark below Definition \ref{def212}, it is easy to see that 
\[[\mcE]_{V_m}(f)=\lim\limits_{n\to\infty}[\mcE_n]_{V_{m,n}}(f_n),\]
where $f_n,f$ are the functions in the statement of (a'). This implies \[\mcE(f)=\lim\limits_{m\to\infty} [\mcE]_{V_m}(f)\leq \liminf\limits_{n\to\infty}\sup\limits_{m\geq 1} [\mcE_n]_{V_{m,n}}(f_n)=\liminf\limits_{n\to\infty}\mcE_n(f_n).\]

(b'). It suffices to consider $f\in \mcF$. Let $V_m$ and $V_{m,n}$ be the same nest sequences in the proof of (a'), and let $f'_n\in C(A_n),n\geq 1$ so that $f_n'\rightarrowtail f$. There exists a sequence $\{m_n\}_{n\geq 1}$ with $m_n\to\infty$ that grows slowly enough so that 
\[
\mcE(f)=\lim\limits_{m\to\infty}[\mcE]_{V_m}(f)=\lim\limits_{n\to\infty}[\mcE_n]_{V_{m_n,n}}(f_n')=\lim\limits_{n\to\infty} \mcE_{n}(f_n), 
\]
where $f_n$ is the harmonic extension of $f_n'|_{V_{m_n,n}}$ on $A_n$ with respect to $(\mcE_n,\mcF_n)$ for each $n\geq 1$. Since $\{R_n\}_{n\geq1}$ is uniformly bounded and equicontinuous and $\{\mcE_n(f_n)\}_{n\geq 1}$ is uniformly bounded, we can see that $\{f_n\}_{n\geq 1}$ is also uniformly bounded and equicontinuous. By applying Lemma \ref{lemma22} and by a routine argument (any subsequence contains a further subsequence that converges to a same function), we have $f_n\rightarrowtail f$. 
\end{proof}

\noindent\textbf{Remark.} In fact, in Theorem \ref{thm213}, the statement (b') can be strengthened as follows:

(b'') Let $f\in C(A)$ and let $V=\{x_m\}_{m=1}^M$ be a finite subset of $A$. For $1\leq m\leq M$, let $x_{m,n}\in A_n,n\geq 1$ such that $x_{m,n}\to x_m$. Then, there exists a sequence $\{f_n\}_{n\geq 1}$ such that $f_n\in C(A_n),\forall n\geq 1$, $f_n\rightarrowtail f$,
\[\begin{cases}
\mcE(f)=\lim_{n\to\infty} \mcE_n(f_n),\\
f_n(x_{m,n})=f(x_m),\forall n\geq 1,1\leq m\leq M.
\end{cases}\] 
The claim can be easily seen from the proof of (b'). \vspace{0.15cm}

The regularity of the form is immediate (see \cite{ki4}) as we always have $(A,R)$ being a compact space in the discussions. Next, we show that $\Gamma$-convergence on $C(B)$ preserves the local property. 

\noindent \textbf{The local property}: a resistance form $(\mcE,\mcF)$ on $X$ (or more generally a Dirichlet form $(\mcE,\mcF)$ on $L^2(X,\mu)$) is local if $\mcE(f,g)=0$ for any $f,g\in \mcF$ such that $f,g$ have disjoint supports. 
 
\begin{theorem}\label{thm214}
Assume (C1). For each $n\geq 1$, let $(\mcE_n,\mcF_n)$ be a resistance form on $A_n$ with $\mcF_n\subset C(A_n)$; let $(\mcE,\mcF)$ be the resistance forms on $A$ with $\mcF\subset C(A)$. Assume $(\mcE_n,\mcF_n)$ $\Gamma$-converges to $(\mcE,\mcF)$ on $C(B)$, and $(\mcE_n,\mcF_n)$ is local for each $n\geq 1$, then $(\mcE,\mcF)$ is local.
\end{theorem}
\begin{proof}
	Let $f\in C(B)$ such that $f|_A\in \mcF$, we will show that $\mcE(f_+,f_-)=0$, where $f_+=f\vee 0$ and $f_-=f\wedge 0$. Then, by a linear combination, it follows immediately that $\mcE(g,h)=0$ for any $g,h\in \mcF$ with $gh=0$. 
	
	First, $\mcE(f)\geq \mcE(|f|)$ due to the Markov property, so $\mcE(f_+,f_-)\geq 0$. 
	
	Next, we prove $\mcE(f_+,f_-)\leq 0$. By Definition \ref{def211} (b), we can choose $f'_{+,n}\in C(B)$ for each $n\geq 1$ so that 
	\[f'_{+,n}\rightrightarrows f_+\text{ and }\mcE(f_+)=\lim\limits_{n\to\infty}\mcE_n(f'_{+,n}).\]
	Also, define $f'_{-,n}$ in a same way. Without loss of generality, we can assume that $f'_{+,n}|_{A_n}\in \mcF_n,f'_{-,n}|_{A_n}\in \mcF_n$ for all $n\geq 1$. There is a decaying sequence of positive numbers $\{\varepsilon_n\}_{n\geq 1}, \varepsilon_n\to 0$ such that  $f_{+,n}=(f'_{+,n}-\varepsilon_n)\vee 0$ and $f_{-,n}=(f'_{-,n}+\varepsilon_n)\wedge 0$ have disjoint supports for each $n\geq 1$. Clearly, $f_{+,n}\rightrightarrows f_+$ and $f_{-,n}\rightrightarrows f_-$. Then, 
	\[
	\begin{aligned}
	\mcE(f)\leq \liminf_{n\to\infty}\mcE_n(f_{+,n}+f_{-,n})&=\liminf_{n\to\infty}\big(\mcE_n(f_{+,n})+\mcE_n(f_{-,n})\big)\\
	&\leq \lim_{n\to\infty}\big(\mcE_n(f'_{+,n})+\mcE_n(f'_{-,n})\big)=\mcE(f_+)+\mcE(f_-),
	\end{aligned}
	\] 
	where we use  Definition \ref{def211} (a) in the first inequality, the local property of $(\mcE_n,\mcF_n)$ in the first equality and the Markov property in the second inequality. This implies that $\mcE(f_+,f_-)\leq 0$. 
\end{proof}

In fact, it suffices to assume the form $(\mcE,\mcF)$ is regular on $L^2(A,\mu)$. See Section 7 for further discussions.

\subsection{Convergence of resolvent kernels} Lastly, we have a short proof showing that $R_n\rightarrowtail R$ implies $u^{(\alpha)}_n\rightarrowtail u^{(\alpha)}$, where $u^{(\alpha)}_n$ and $u^{(\alpha)}$ are the resolvent kernels associated with $R_n,\mu_n$ and $R,\mu$ defined as follows.

\begin{definition}\label{def215}
Let $\alpha>0$, let $R$ be a bounded resistance metric on $X$, let $\nu$ be a probability measure on $(X,R)$, and assume $(X,R)$ is separable. Then for any $x$, there exists a unique function $u^{(\alpha)}(x,\bullet)\in \mcF$ such that 
\[\mcE^{(\alpha)}\big(u^{(\alpha)}(x,\bullet),f\big)=f(x),\quad\forall f\in \mcF,\]
where $(\mcE,\mcF)$ is the resistance form associated with $R$, and $\mcE^{(\alpha)}(f,g)=\mcE(f,g)+\alpha\int_Xfgd\nu$ for any $f,g\in \mcF$. The function $u^{(\alpha)}$ is the \emph{resolvent kernel} associated with $R$ and $\nu$. 
\end{definition}

It is well-known that $u^{(\alpha)}$ is a symmetric function, i.e. $u^{(\alpha)}(x,y)=u^{(\alpha)}(y,x)$. Following a routine argument, we have the following estimate.

\begin{lemma}\label{lemma216}
	Let $\alpha>0$, let $R$ be a bounded resistance metric on $X$, let $\nu$ be a probability measure on $X$, and assume that $(X,R)$ is separable. Then there exists $c$ depending only on $\alpha$ and $\diam_R(X)=\sup_{x,y\in X}R(x,y)$ such that 
	\[\begin{cases}
		\|u^{(\alpha)}\|_{C(X^2)}\leq c,\\
		|u^{(\alpha)}(x_1,y_1)-u^{(\alpha)}(x_2,y_2)|\leq c\big(R(x_1,x_2)^{1/2}+R(y_1,y_2)^{1/2}\big),\quad\forall x_i,y_i\in X,i=1,2.
	\end{cases}\]
\end{lemma}
\begin{proof}
By definition, we have $\mcE\big(u^{(\alpha)}(x,\bullet)\big)\leq \mcE^{(\alpha)}\big(u^{(\alpha)}(x,\bullet)\big)=u^{(\alpha)}(x,x)$, where $(\mcE,\mcF)$ is the resistance form associated with $R$. As a consequence, we have 
\[
|u^{(\alpha)}(x,y_1)-u^{(\alpha)}(x,y_2)|^2\leq R(y_1,y_2) u^{(\alpha)}(x,x).
\]
Combining with the fact that $\int_X u^{(\alpha)}(x,y)d\nu(y)=\alpha^{-1}$, we have \[u^{(\alpha)}(x,x)-\big(u^{(\alpha)}(x,x)\big)^{1/2}\diam^{1/2}_R(X)\leq \alpha^{-1}.\]
This shows that $u^{(\alpha)}(x,x)\leq c'$ for some $c'$ depending only on $\diam_R(X)$ and $\alpha$. 
\end{proof}

\begin{theorem}\label{thm217}
Assume (C1),(C2). Let $R_n\in \mathcal{RM}(A_n)\cap C(A_n^2)$ for each $n\geq 1$, and assume $R_n\rightarrowtail R$ for some $R\in \mathcal{M}(A)$. Then $R\in \mathcal{RM}(A)\cap C(A^2)$. Let $\alpha>0$, then 
\[u^{(\alpha)}_n\rightarrowtail u^{(\alpha)},\]
where $u_n^{(\alpha)}$ is the resolvent kernel associated with $R_n,\mu_n$ for each $n\geq 1$, and $u^{(\alpha)}$ is the resolvent kernel associated with $R,\mu$.
\end{theorem}
\begin{proof}
	We follow a similar idea as the proof of Theorem 2.4.1 in \cite{Mosco}. Let $(\mcE_n,\mcF_n)$ be the resistance form associated with $R_n$ for $n\geq 1$, and let $(\mcE,\mcF)$ be the resistance form associated with $R$. Then, $u^{(\alpha)}_n(x,\cdot)$ is the minimizer of $\mcE^{(\alpha)}_n(f)-2f(x)$ for $n\geq 1$, $x\in A_n$; $u^{(\alpha)}_n(x,\cdot)$ is the minimizer of $\mcE^{(\alpha)}(f)-2f(x)$, for $x\in A$. 
	
	By Lemma \ref{lemma216}, the sequence $\{u^{(\alpha)}_n\}_{n\geq 1}$ is uniformly bounded and equicontinuous, so there is a subsequence $n_k\to\infty$ and $u'\in C(A^2)$ such that  $u^{(\alpha)}_{n_k}\rightarrowtail u'$. Let $x_n\in A_n,n\geq 1$ and $x_n\to x\in A$. By Theorem \ref{thm213} and Lemma \ref{lemma25}, we have 
	\[\mcE^{(\alpha)}\big(u'(x,\cdot)\big)-2u'(x,x)\leq \liminf\limits_{k\to\infty}\mcE^{(\alpha)}_{n_k}\big(u^{(\alpha)}_{n_k}(x_{n_k},\cdot)\big)-2u^{(\alpha)}_n(x_{n_k},x_{n_k}).\]
	On the other hand, by Theorem \ref{thm213} and Lemma \ref{lemma25}, there is a sequence $f_n\rightarrowtail u^{(\alpha)}(x,\cdot)$ such that 
	\[\mcE^{(\alpha)}\big(u^{(\alpha)}(x,\cdot)\big)=\lim_{n\to\infty}\mcE^{(\alpha)}_n(f_n).\]
	Thus, 
	\[\begin{aligned}
	\mcE^{(\alpha)}\big(u^{(\alpha)}(x,\cdot)\big)-2u^{(\alpha)}(x,x)&=\lim_{n\to\infty}\mcE^{(\alpha)}_n(f_n)-2f_n(x_n)\\
	&\geq \liminf\limits_{k\to\infty}\mcE^{(\alpha)}_{n_k}\big(u^{(\alpha)}_{n_k}(x_{n_k},\cdot)\big)-2u^{(\alpha)}_{n_k}(x_{n_k},x_{n_k})\\
	&\geq \mcE^{(\alpha)}\big(u'(x,\cdot)\big)-2u'(x,x).
	\end{aligned}\]
	As $x\in A$ is arbitrary, this implies that $u'=u^{(\alpha)}$. The argument works for any subsequence, so the theorem follows.  
\end{proof}

Careful readers may naturally expect consequences of Theorem \ref{thm217} like the convergence of spectrum. In fact, in \cite{KS}, the Mosco convergence on converging spaces is defined and several consequences are provided. Theorem \ref{thm217} implies such convergence of the form. In addition, readers can find useful discussions on different convergence concepts of metric measure spaces in  \cite{BBI,GMS}.

\subsection{Related results in stochastic process} 
Our discussions are purely analytic, and our interest is to find a limit resistance metric with Theorem \ref{thm29}, with Theorem \ref{thm213} serving as a tool to provide additional information on the limit form (so we can prove the self-similarity in Section 6.1).

On the other hand, for readers interested in deeper consequences of Theorem \ref{thm29}, there is the paper \cite{C1} by D.A. Croydon in the probability setting considering a question similar to Theorem \ref{thm213}, which shows that under weak assumptions, which will be verified in our setting in the rest of this subsection,
\begin{equation}\label{eqn21}
\mathbb{P}^n_{x_n}(X_t^n\in \cdot)\Rightarrow \mathbb{P}_x(X_t\in \cdot),\qquad\forall x_n\to x,
\end{equation}
with respect to the the usual Skorohod $J_1$-topology on the space of cadlag process, where $(\mathbb{P}^n_x,x\in A_n,X_t^n,t\geq 0)$ and $(\mathbb{P}_x,x\in A,X_t,t\geq 0)$ are the Hunt processes associated with the Dirichlet forms $(\mcE_n,L^2(A_n,\mu_n))$ and $(\mcE,L^2(A,\mu))$ (see \cite{FOT}). Also, see \cite{CHK} for a stronger result about convergence of local times under slightly stronger assumptions.

There are subtle differences between the basic setting of \cite{C1,CHK} and our assumptions in Theorem \ref{thm213}. To apply the main theorem of \cite{C1}, we need to verify the basic geometric assumption (Assumption 1.1 (a) of \cite{C1} or Assumption 1.3 of \cite{CHK}), which states that the pointed metric measure spaces $(A_n,R_n,\mu_n,x_n)$ converge to $(A,R,\mu,x)$ in the Gromov–Hausdorff-vague topology (See section 2.2 of \cite{C1}). We provide a short proof here. 

Under the assumption of Theorem \ref{thm213}, we can embed all the resistance spaces $(A,R)$ and $(A_n,R_n),n\geq 1$ into a new metric space $(\bar{A},\bar{R})$ to be defined as follows. First, we let $\bar{A}$ be the same as in Proposition \ref{prop23} (b) (but we view $\bar{A}$ as the disjoint union of $A$ and $A_n,n\geq 1$ for convenience)
\[\bar{A}=A\sqcup \big(\sqcup_{n\geq 1} A_n\big)\cong (\{0\}\times A)\cup\big(\cup_{n\geq 1}(\{\frac{1}{n}\}\times A_n)\big).\]
Next, by Proposition \ref{prop23} (b), the fact that $R_n\rightarrowtail R$ and $\bar{A}$ is compact (induced topology from $[0,1]\times B$), we can find a continuous increasing function $\psi:[0,\infty)\to [0,\infty)$ so that $\psi(0)=0$ and $|R_n(x,x')-R(y,y')|\leq \psi\big(\frac{1}{n}+d(x,y)\big)+\psi\big(\frac{1}{n}+d(x',y')\big)$, for any $n\geq 1$, $x,x'\in A_n$ and $y,y'\in A$. We define the metric $\bar{R}$ on $\bar{A}$ as follows,
\[\bar{R}(x,y)=\begin{cases}
	R_n(x,y),\qquad &n\geq 1, x,y\in A_n,\\
	\inf\limits_{x'\in A_n,y'\in A}\Big(R_n(x,x')+R(y,y')+\psi\big(\frac{1}{n}+d(x',y')\big)\Big),&n\geq 1, x\in A_n, y\in A,\\
	\inf\limits_{z\in A}\big(\bar{R}(x,z)+\bar{R}(z,y)\big),&n\neq m, x\in A_n,y\in A_m.
\end{cases}\] 
It is not hard to show that $(\bar{A},\bar{R})$ is a metric space. In addition, let $\iota_n:A_n\to \bar{A},n\geq 1$ and $\iota:A\to \bar{A}$ be the natural embeddings, one can easily prove that $\lim_{n\to\infty}\delta\big(\iota(A_n),\iota(A)\big)=0$ where $\delta$ is the Hausdorff metric on subsets of $(\bar{A},\bar{R})$, and $\mu_n\circ \iota_n^{-1}\Rightarrow \mu\circ \iota^{-1}$. For short, we omit the details here. This immediately implies that  $(A_n,R_n,\mu_n,x_n)$ converges to $(A,R,\mu,x)$ in the Gromov–Hausdorff-vague topology. 

In particular, by Theorem 1.2 of \cite{C1} and our particular embedding, one can see (\ref{eqn21}) holds with all the process embedded in the original topological space $(B,d)$.

\section{A review of $\AG_\lambda$}
In this section, we introduce some observations and notations about $\AG_\lambda$. At the end, we will introduce the main results we will prove in Section 4-6. 

\begin{definition}\label{def31}
	(a). Let $W_0=\emptyset$ and $W_m=\{1,2,3,4\}^m$ for $m\geq 1$. Write $W_*=\bigcup_{m\geq 0}W_m$. 
	
	(b). Write $F_{w,\lambda}=F_{w_1,\lambda}F_{w_2,\lambda}\cdots F_{w_m,\lambda}$ for $w=w_1w_2\cdots w_m\in W_m$, $m\geq 1$. Also, define $F_{\emptyset,\lambda}=id$, the identity map.
\end{definition}

We use the notation $\blacktriangle$ for the triangle with vertices $p_1,p_2,p_3$. Also, $\partial\blacktriangle$ is the topological boundary of $\blacktriangle$ in $\mathbb{R}^2$, and $\partial_\downarrow \blacktriangle$ is the bottom line of the triangle, i.e. $\partial_\downarrow\blacktriangle=\{(t,0):0\leq t\leq 1\}$. 

\begin{proposition}\label{prop32}
	Let $w\in W_*$ and $\lambda_1,\lambda_2\in (0,\frac{1}{2})$. Then for any closed set $A\subset \blacktriangle$, we have 
	\[\delta(F_{w,\lambda_1}A,F_{w,\lambda_2}A)\leq 2|\lambda_1-\lambda_2|.\]
	In particular, for $i=1,2,3$, we have $d\big(F_{w,\lambda_1}(p_i),F_{w,\lambda_2}(p_i)\big)\leq 2|\lambda_1-\lambda_2|$. 
\end{proposition}
\begin{proof}
	By an easy observation, we have 
	\[|\lambda_1-\lambda_2|=\sup\{x\in \blacktriangle:d\big(F_{4,\lambda_1}(x),F_{4,\lambda_2}(x)\big)\}.\]
	Then, for any compact set $S_1,S_2\subset \blacktriangle$, we can easily see that for $1\leq i\leq4$,
	\[
	\begin{aligned}
	\delta(F_{i,\lambda_1}(S_1),F_{i,\lambda_2}(S_2))&\leq  \delta(F_{i,\lambda_1}(S_1),F_{i,\lambda_1}(S_2))+ \delta(F_{i,\lambda_1}(S_2),F_{i,\lambda_2}(S_2))\\
	&\leq \frac{1}{2}\delta(S_1,S_2)+|\lambda_1-\lambda_2|,
	\end{aligned}
	\]
	As a consequence, for $w\in W_m$, we have 
	\[
	\begin{aligned}
	\delta(F_{w,\lambda_1}A, F_{w,\lambda_2}A)&\leq \frac{1}{2}\delta(F_{w_2\cdots w_m,\lambda_1}A, F_{w_2\cdots w_m,\lambda_2}A)+|\lambda_1-\lambda_2|\\
	&\leq \frac{1}{4}\delta(F_{w_3\cdots w_m,\lambda_1}A, F_{w_3\cdots w_m,\lambda_2}A)+\frac{1}{2}|\lambda_1-\lambda_2|+|\lambda_1-\lambda_2|\\&\leq\cdots\leq  2|\lambda_1-\lambda_2|.
	\end{aligned}
	\]
	The second claim follows from  $d\big(F_{w,\lambda_1}(p_i),F_{w,\lambda_2}(p_i)\big)=\delta\big(\{F_{w,\lambda_1}(p_i)\},\{F_{w,\lambda_2}(p_i)\}\big)$.
\end{proof}

An easy observation from Proposition \ref{prop32} is that $\AG_\lambda$ depends on $\lambda$ in a continuous manner.

\begin{corollary}\label{coro33}
	We have $\delta(\AG_{\lambda_1},\AG_{\lambda_2})\leq 2|\lambda_1-\lambda_2|$ for $\lambda_1,\lambda_2\in (0,\frac{1}{2})$. 
\end{corollary}
\begin{proof}
	We have 
	\[\begin{aligned}
	\delta(\AG_{\lambda_1},\AG_{\lambda_2})&=\lim_{m\to\infty}\delta(\bigcup_{w\in W_m}F_{w,\lambda_1}\blacktriangle, \bigcup_{w\in W_m}F_{w,\lambda_2}\blacktriangle)\\
	&\leq \sup_{m\geq 0}\sup_{w\in W_m} \delta(F_{w,\lambda_1}\blacktriangle, F_{w,\lambda_2}\blacktriangle)\leq 2|\lambda_1-\lambda_2|,
	\end{aligned}\]
	as $\AG_\lambda=\lim_{m\to\infty}\big(\bigcup_{w\in W_m}F_{w,\lambda}\blacktriangle\big)$ for any $0<\lambda<\frac{1}{2}$.
\end{proof}

Next, we talk about the approximating graphs of $\AG_\lambda$. 

\begin{definition}\label{def34}
	We define a connected graph $G_{m,\lambda}=\big(V_{m,\lambda},E_{m,\lambda}\big)$ as follows. 
	
	(a). Let $V_0=V_{0,\lambda}=\{p_1,p_2,p_3\}$. For $m\geq 1$, we define $V_{m,\lambda}=\bigcup_{w\in W_m} F_{w,\lambda}V_0.$
	
	(b). For $m\geq 0$, define 
	\[E_{m,\lambda}=\big\{\{x,y\}\subset V_{m,\lambda}:x\neq y,\text{ there exists }w\in W_m\text{ such that  }\{x,y\}\subset F_{w,\lambda}\AG_\lambda\big\}.\]
	
We also write $V_{*,\lambda}=\bigcup_{m=0}^\infty V_{m,\lambda}$.
\end{definition}

See Figure \ref{fig31} for an illustration of the approximating graphs $G_{m,\lambda}$.

\begin{figure}[htp]
	\includegraphics[width=4.5cm]{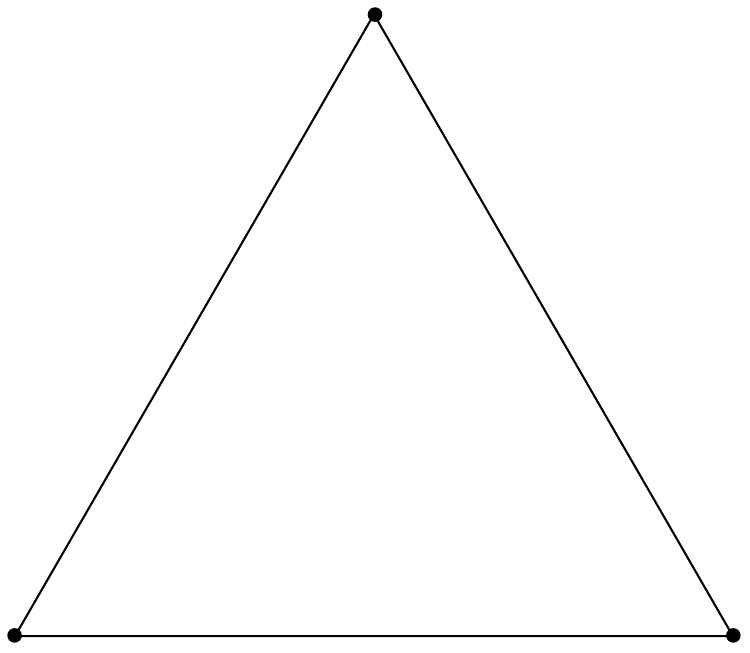}\qquad
	\includegraphics[width=4.5cm]{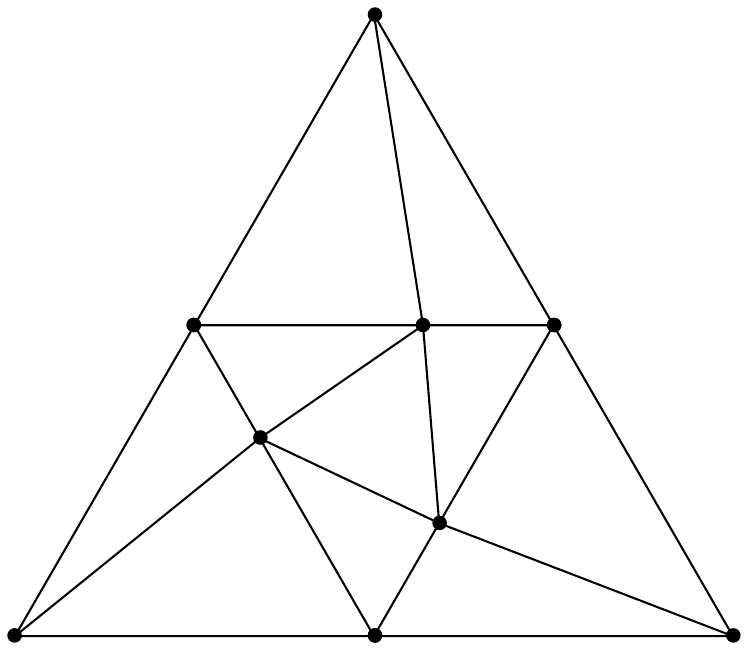}\qquad
	\includegraphics[width=4.5cm]{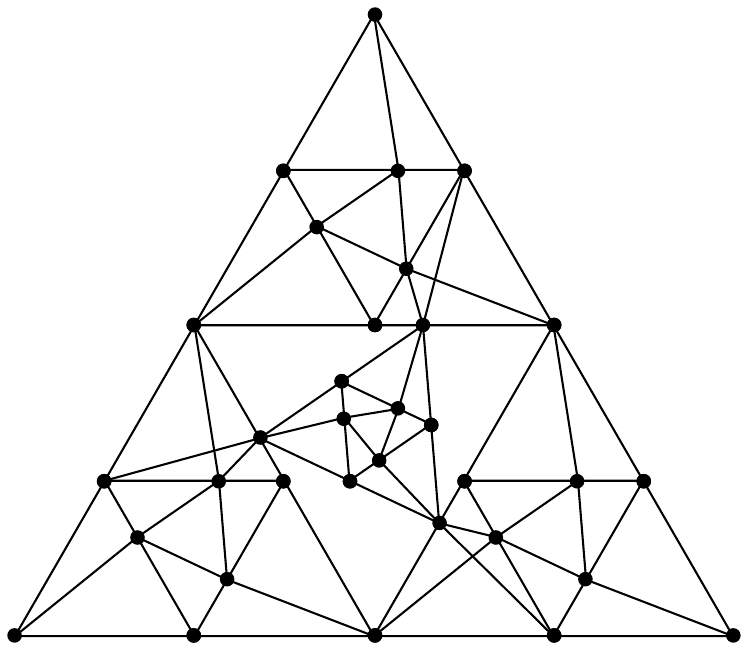}
	\caption{The graphs $G_{0,\lambda}, G_{1,\lambda}$ and $G_{2,\lambda}$.}\label{fig31}
\end{figure}

Our construction here is an example of the finitely ramified cell structure \cite{T}. However, $\AG_\lambda$ is not a finitely ramified self-similar set in the sense of Definition 7.1 of \cite{T} when $\lambda\notin \mathbb{Q}$. In fact, let 
\[
\tilde{V}_{0,\lambda}=\bigcup_{m=0}^\infty\bigcup_{w\in W_m} F_{w,\lambda}^{-1}\big(V_{m,\lambda}\cap F_{w,\lambda}\AG_{\lambda}\big).
\]
It is not hard to see that $\# \tilde{V}_{0,\lambda}<\infty$ if and only if $\lambda\in \mathbb{Q}\cap (0,\frac{1}{2})$. See Figure \ref{fig32} for examples of $\tilde{V}_{0,\lambda}$, $\lambda\in \mathbb{Q}$. 

\begin{figure}[htp]
	\includegraphics[width=4cm]{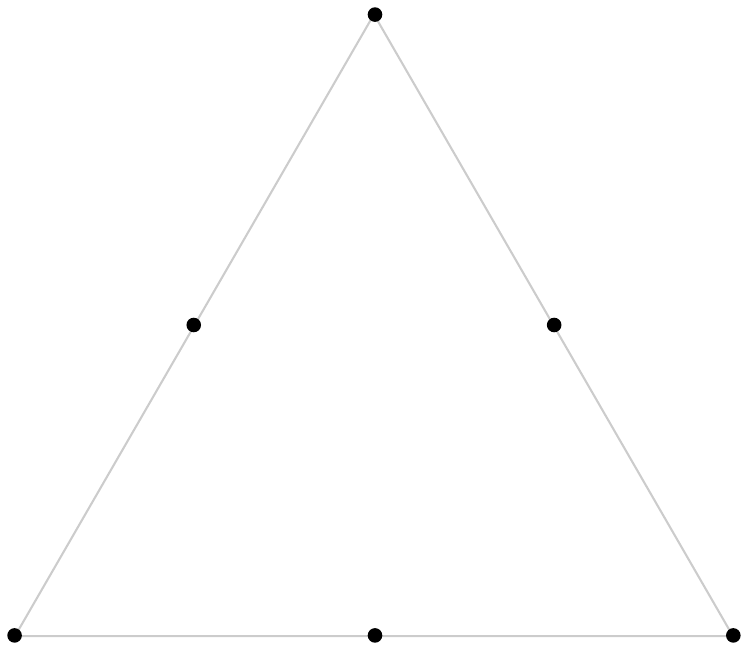}\qquad
	\includegraphics[width=4cm]{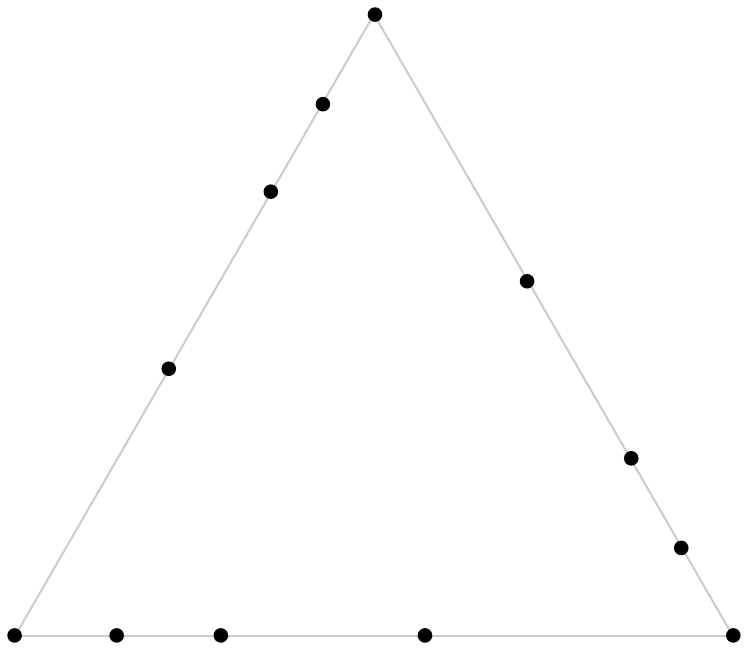}\qquad
	\includegraphics[width=4cm]{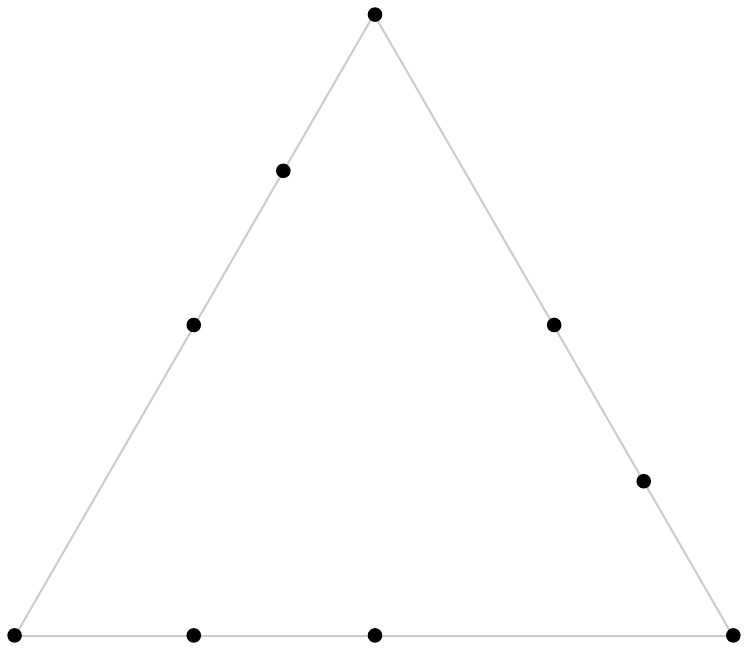}\vspace{0.1cm}
	\begin{picture}(0,0)
	\put(-350,-10){$\lambda=1/4$}
	\put(-212,-10){$\lambda=1/7$}
	\put(-78,-10){$\lambda=1/8$}
	\end{picture}
	\caption{The vertice set $\tilde{V}_{0,\lambda}$.}\label{fig32}
\end{figure}
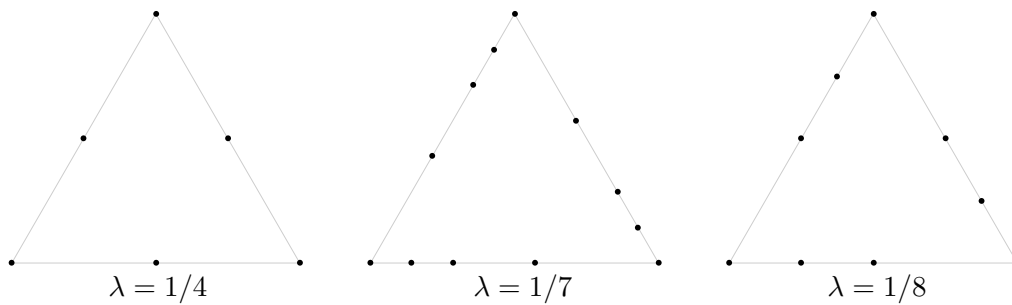

$\AG_\lambda$ is also not p.c.f. when $\lambda\notin \mathbb{Q}$. In fact,  $\AG_\lambda$ is p.c.f. if and only if $\#\pi^{-1}_\lambda(\tilde{V}_{0,\lambda})<\infty$, where $\pi_\lambda:\{1,2,3,4\}^\infty\to \AG_\lambda$ is defined by
\[\{\pi_\lambda(\omega)\}=\bigcap_{m=1}^\infty F_{[\omega]_m}\AG_\lambda,\quad \forall \omega=\omega_1\omega_2\cdots\in \{1,2,3,4\}^\infty,\]
where $[\omega]_m=\omega_1\omega_2\cdots \omega_m$ for $m\geq 1$. See the book \cite{ki3} for the full definition of p.c.f. self-similar sets. In fact, the class $\AG_\lambda,\lambda\in (0,\frac{1}{2})$ was raised in the book \cite{B} to show the delicacy of the conditions defining p.c.f. self-similar sets.

\subsection{Main results}
The following classes will be the main objects that we study in the rest of this paper. 

\begin{definition}\label{def35}
	Let $0<r,s<1$. We say a resistance form $(\mcE,\mcF)$ on $\AG_\lambda$ is self-similar with renormalization factors $r_1=r_2=r_3=r$ and $r_4=s$ if $f\in \mcF$ implies $f\circ F_{i,\lambda}\in \mcF,\forall 1\leq i\leq4$ and
	\[
	\mcE(f)=\sum_{i=1}^4 r_i^{-1}\mcE(f\circ F_{i,\lambda}). 
	\]
	Let $\mathcal{G}$ be the group of rotations that acts as $A_3$ on $V_0$. We say the form $(\mcE,\mcF)$ is $\mathcal{G}$-symmetric if $f\in \mathcal{F}$ implies $f\circ \sigma\in \mathcal{F}$ and $\mcE(f)=\mcE(f\circ \sigma)$ for any $\sigma\in \mathcal{G}$. In addition, we say $(\mcE,\mcF)$ is regular if  $\mcF$ is a dense subset of $C(\AG_\lambda)$ with respect to the supremum norm.
	
	Let $\mathscr{E}(\lambda,s,r)$ denote the family of regular $\mathcal{G}-$symmetric self-similar resistance forms $(\mcE,\mcF)$ on $\AG_\lambda$ with renormalization factors $r_1=r_2=r_3=r$ and $r_4=s$.
\end{definition}

Our main theorem states that the symmetric self-similar resistance forms are uniquely determined by two coefficients $\lambda,s$, where $\lambda$ determines the topology of the fractal, and $s$ determines the `size' of the added cell $F_{4,\lambda}\AG_\lambda$. 

\begin{theorem}\label{thm36}
	Let $s\in (0,1)$ and $\lambda\in (0,\frac{1}{2})$. Then there is a unique $r\in [\frac{3}{5},1)$ such that $\mathscr{E}(\lambda,s,r)\neq\emptyset$. In addition, the regular $\mathcal{G}$-symmetric self-similar resistance form $(\mathcal{E},\mathcal{F})\in \mathscr{E}(\lambda,s,r)$ is unique up to scalar multiples. 
\end{theorem}

\noindent\textbf{Remark.} The above results can be extended to $s\in (0,\infty)$ with the following modifications.

\noindent1. Define $\mathscr{E}(\lambda,s,r)$ to be the set of self-similar resistance forms on $\bigcup_{w\in W_*}F_{w,\lambda}\partial\blacktriangle$.

\noindent2. With a suitable measure $\mu$, $(\mcE,\mcF)$ becomes a local regular irreducible conservative Dirichlet form on $L^2(\AG_\lambda,\mu)$. 

See Section 7 for some supplements.\vspace{0.15cm}

In the rest of this paper, we will use the following notations for convenience.

\begin{definition}\label{def37}
	Let $0<\lambda<\frac{1}{2}$ and $0<s<1$. By Theorem \ref{thm36}, we can introduce the following notations. 
	
	(a). We let $r(\lambda,s)$ be the unique number in $(0,1)$ such that $\mathscr{E}(\lambda,s,r(\lambda,s))\neq \emptyset$. 
	
	(b). We define $(\mcE_{\lambda,s},\mcF_{\lambda,s})$ to be the unique form in $\mathscr{E}(\lambda,s,r(\lambda,s))$ such that 
	\[R_{\lambda,s}(p_1,p_2)=R_{\lambda,s}(p_2,p_3)=R_{\lambda,s}(p_3,p_1)=\frac{2}{3},\]
	where we also define
	\[R_{\lambda,s}(x,y)=\sup\big\{\frac{|f(x)-f(y)|^2}{\mcE_{\lambda,s}(f)}:f\in \mcF_{\lambda,s}\setminus Constants\big\}.\]
\end{definition}

We will see that $\mcE_{\lambda,s}$ changes continuously as a consequence of the uniqueness theorem in Theorem \ref{thm36}.
\begin{theorem}\label{thm38}
	Let $\{\lambda_n\}_{n\geq 1}\subset (0,\frac{1}{2})$ and $\{s_n\}_{n\geq 1}\subset (0,1)$. If $\lambda_n\to \lambda\in (0,\frac{1}{2})$ and $s_n\to s\in (0,1)$, then $R_{\lambda_n,s_n}\rightarrowtail R_{\lambda,s}$ and clearly $(\mcE_{\lambda_n,s_n},\mcF_{\lambda_n,s_n})$ $\Gamma$-converges to $(\mcE_{\lambda,s},\mcF_{\lambda,s})$ on $C(\blacktriangle)$. 
\end{theorem}

As a consequence of Theorem \ref{thm217}, we also have the following. 
\begin{corollary}\label{coro39}
	Let $\alpha>0$, and let $\{\mu_\lambda\}_{\lambda\in (0,\frac{1}{2})}$ be a class of probability measures on $\blacktriangle$, with $\AG_\lambda$ being the support of $\mu_\lambda$ for each $\lambda\in (0,\frac{1}{2})$. Also, for $\lambda\in (0,\frac{1}{2}),s\in (0,1)$, let $u^{(\alpha)}_{\lambda,s}\in C(\AG_\lambda\times \AG_{\lambda})$ being the resolvent kernel associated with $(\mcE_{\lambda,s},\mcF_{\lambda,s})$, i.e. 
	\[f(x)=\mcE\big(u^{(\alpha)}_{\lambda,s}(x,\cdot),f\big)+\alpha\int_{\AG_\lambda} f(y)u^{(\alpha)}_{\lambda,s}(x,y)d\mu_\lambda(y),\quad \forall f\in \mcF_{\lambda,s}.\]
	If $\mu_\lambda$ is continuous in $\lambda$ with respect to the weak convergence topolgy on $\blacktriangle$. Then  $u^{(\alpha)}_{\lambda,s}(x,y)$ is continuous on $(0,1)\times \wwAG$, where 
	\[\wwAG=\{(\lambda,x,y):\lambda\in (0,\frac{1}{2}), x,y\in \AG_\lambda\},\]
	and we view $u^{(\alpha)}:(s,\lambda,x,y)\to \mathbb{R}$.
\end{corollary}

\section{Dyadic rational cases}
We denote the set of dyadic rationals by $\mathbb{D}$, i.e.
\[\mathbb{D}=\{\frac{k}{2^m}:k\in \mathbb{Z},m\in \mathbb{N}\}.\]
Also, for $m\geq 1$, define 
\[\mathbb{D}_m=\{\frac{k}{2^m}:k\in \mathbb{Z}\}.\]
We view $\AG_\lambda,\lambda\in \mathbb{D}$ as the simplest cases, in that $\tilde{V}_{0,\lambda}\subset \bigcup_{w\in \{1,2,3\}^m}F_{w,\lambda}V_0$ for some $m\geq 1$ (when $\lambda\in \mathbb{D}_{m+1}$). In particular, $\AG_\lambda$ is p.c.f. when $\lambda\in \mathbb{D}$.

It is well-known that the existence of a self-similar resistance form on a p.c.f. self-similar set is equivalent to the existence of a solution to a fixed point problem of certain electrical networks \cite{ki3}. Let's briefly recall the fact below.

For each resistance form $\mcD$ on $\tilde{V}_{0,\lambda}$ and $\bm{r}=\{r_i\}_{i=1}^4\in \mathbb{R}_+^4$, we define $\Lambda_r\mcD$ on $\tilde{V}_{0,\lambda}$ by
\[\Lambda_{\bm{r}}\mcD(f)=\min\{\sum_{i=1}^4 r_{i}^{-1}\mcD(g\circ F_{i,\lambda}):g\in l(\tilde{V}_{1,\lambda}),g|_{\tilde{V}_{0,\lambda}}=f\},\quad\forall f\in l(\tilde{V}_{0,\lambda}),\]
where we have
\[\tilde{V}_{m,\lambda}=\bigcup_{w\in W_m}F_{w,\lambda}\tilde{V}_{0,\lambda}\] for $m\geq1$. For $r,s\in (0,1)$, 
$\mathscr{E}(\lambda,s,r)\neq \emptyset$ if and only if there is $\mcD\in \mathcal{RF}(\tilde{V}_{0,\lambda})$ such that $\mcD=\Lambda_{\bm{r}}\mcD$, with $r_1=r_2=r_3=r$ and $r_4=s$. 

The fixed point problem has been studied in rich contexts \cite{M1,M2,Mosco,Pe}. In particular, Sabot first solved the uniqueness problem of resistance forms on nest fractals in his celebrated work \cite{Mosco}. He introduced the concept of preserved $\mathcal{G}-$relation to describe possible degenerated fixed points, where $\mathcal{G}$ is a symmetry group acting on the fractal.  We will use partial results from Sabot's paper \cite{Sabot}, which are translated into Definition \ref{def41} and Theorem \ref{thm42} as follows.

\begin{definition}\label{def41}
	Let $\mathcal{J}$ be an equivalence relation on $\tilde{V}_{0,\lambda}$. Recall that $\mathcal{G}$ is the group of isometries of $\AG_\lambda$ that acts as the permutation group $A3$ on $V_0$.
	
	(a). Define $\mathcal{J}^{(1)}$ to be the smallest equivalence relation on $\tilde{V}_{1,\lambda}$ such that
	\[x \mathcal{J} y\Longrightarrow F_{i,\lambda}(x) \mathcal{J}^{(1)}F_{i,\lambda}(y), \quad 1\leq i\leq 4.\]
	
	(b). Call $\mathcal{J}$ a preserved $\mathcal{G}-$relation if for any $x,y \in \tilde{V}_{0,\lambda}$,
	\[x \mathcal{J} y\Longrightarrow \sigma(x) \mathcal{J} \sigma(y),\quad \forall \sigma\in \mathcal G,\]
	and
	\[x \mathcal{J} y\Longleftrightarrow x \mathcal{J}^{(1)}y.\]
\end{definition}

We say that a $\mathcal{G}-$preserved relation $\mathcal{J}$ is trivial, if either $\mathcal{J}$ is the full relation ($\mathcal{J}=1$, i.e. $x\mathcal{J} y,\forall x,y\in \tilde{V}_{0,\lambda}$) or the empty relation ($\mathcal{J}=0$, i.e. $x\mathcal{J} y$ implies $x=y$).

\begin{theorem}[\cite{Sabot}]\label{thm42}
	Let $\bm{r}=\{r_i\}_{i=1}^4$ with $r_1=r_2=r_3\in (0,\infty)$ and $r_4\in (0,\infty)$. If there is no non-trivial preserved $\mathcal G$-relation, then we have exactly one (up to scale multiples) $\mathcal G$-symmetric $\mcD\in \mathcal{RF}(\tilde{V}_{0,\lambda})$ (i.e. $\mcD(f)=\mcD(f\circ\sigma)$ for any $f\in l(\tilde{V}_{0,\lambda})$ and $\sigma\in \mathcal{G}$) such that 
	\[\Lambda_{\bm{r}}\mcD=C\mcD,\]
	where $C$ is a constant uniquely determined by $\bm{r}$.
\end{theorem}

It is in general hard to find all the preserved $\mathcal{G}$-relations. One useful idea is to look at $\tilde{V}_{m,\lambda}$, which reflects more information about the fractal than $\tilde{V}_{1,\lambda}$. See \cite{CHQES} for an application of this idea on a class of Julia sets. 

Let $G=(V,E)$ and $G'=(V',E')$ be two finite graphs, and $f: V\to V'$. We define $f(G)=(f(V),f(E))$ to be the graph with vertices $f(V)$ and edges \[{f(E)=\big\{\{f(x),f(y)\}:\{x,y\}\in E,f(x)\neq f(y)\big\}.}\]
We also define $G\cup G'$ {to be the} graph with vertices $V\cup V'$ and edges $E\cup E'$. We can extend the definition of $\mathcal{J}^{(1)}$ to $\mathcal{J}^{(k)}$ for $k\geq1$.

\begin{definition}\label{def43}
	Let $\mathcal J$ be an equivalence relation on $\tilde{V}_{0,\lambda}$.
	
	(a). Define $G_{\mathcal J}=(\tilde{V}_{0,\lambda},E_{\mathcal J})$ to be the graph with vertices $\tilde{V}_{0,\lambda}$ and edges
	\[E_{\mathcal J}=\big\{\{x,y\}\subset \tilde{V}_{0,\lambda}\times \tilde{V}_{0,\lambda}: x\neq y, x\mathcal{J} y\big\}.\]
	
	(b). For $k\geq 0$, we define the graph $G_{\mathcal J}^{(k)}(=(\tilde{V}_{k,\lambda},E_{\mathcal J}^{(k)}))=\bigcup_{w\in W_k}F_{w,\lambda}(G_{\mathcal J})$. 
	
	(c). We define an equivalence relation ${\mathcal J}^{(k)}$ on $\tilde{V}_{k,\lambda}$ by
	\[x{\mathcal J}^{(k)}y\Longleftrightarrow {x \text{ and }y\text{ belong to the same connected component of $G^{(k)}_{\mathcal J}$}}.\]
\end{definition}

The definitions of $G_\mathcal{J}$ and $G^{(k)}_{\mathcal{J}}$ acctually depend on $\lambda\in \mathbb{D}$. Since we always fix a $\lambda$ in this section, we omit the index `$\lambda$' for convenience.

It is not had to show that when $k=1$, Definition \ref{def43} (c) agrees with Definition \ref{def41} (a). In addition, it is not hard to prove the following lemma. See Lemma 3.10 in \cite{CHQES} for details.

\begin{lemma}\label{lemma44}
	Let $k\geq 1$ and $\mathcal{J}$ be a preserved $\mathcal{G}$-relation. We have
	\[x\mathcal{J} y\Longleftrightarrow x\mathcal{J}^{(k)}y,\qquad\forall x,y\in \tilde{V}_{0,\lambda}.\]
\end{lemma}

\begin{proposition}\label{prop45}
	There is no non-trivial preserved $\mathcal{G}$-relation on $\tilde{V}_{0,\lambda}$ for $\lambda\in \mathbb{D}\cap (0,\frac{1}{2})$.  
\end{proposition}
\begin{proof}
	Without loss of generality, we assume that $\lambda\in \mathbb{D}_m$ for some $m\geq 2$. Then,  
	\begin{equation}\label{eqn41}
	\tilde{V}_{0,\lambda}\subset \big(\bigcup_{w\in \{1,2,3\}^{m-1}}F_{w,\lambda}V_0\big)\cap \partial\blacktriangle.
	\end{equation}
	Let $\mathcal{J}$ be a preserved $\mathcal{G}-$relation on $\tilde{V}_{0,\lambda}$. We will show that either $\mathcal{J}=0$ or $\mathcal{J}=1$.\vspace{0.15cm}
	
	\noindent \textit{Case 1: $x\mathcal{J} y$ for any $x,y\in V_0$.} In this easy case, $x\mathcal{J}^{(m-1)} y$ for any $x,y\in \bigcup_{w\in \{1,2,3\}^{m-1}}F_{w,\lambda}V_0$, which implies that $x\mathcal{J} y, \forall x,y\in \tilde{V}_{0,\lambda}$ by Lemma \ref{lemma44} and (\ref{eqn41}). So $\mathcal{J}=1$.\vspace{0.15cm}
	
	\noindent \textit{Case 2: $x{\mathcal{J}\mkern-10.5mu\backslash} y$ if $x\neq y\in V_0$.} In this case, we will see that $\mathcal{J}=0$. 
	
	For a sequence $x_0,x_1,x_2,\cdots,x_N$ such that $\{x_i,x_{i+1}\}\in E_{\mathcal J}^{(n)},\forall 0\leq i\leq N-1$, we call it a $G_{\mathcal J}^{(n)}$-path, where $n\geq 0$. The path is self-avoiding if $x_i\neq x_j,\forall i\neq j\in \{0,1,\cdots,N\}$.
	
	Now, fix $x\in \tilde{V}_{0,\lambda}$, we study a self-avoiding $G_{\mathcal J}^{(m)}$-path $x_0,x_1,\cdots,x_N$ with $x_0=x$. 
	
	First, let
	\begin{equation}\label{eqn42}
	A=\bigcup_{k=0}^{m-1}\bigcup_{w\in \{1,2,3\}^k\times \{4\}}F_{w,\lambda}(\tilde{V}_{m-k-1,\lambda}\setminus V_0).
	\end{equation}
	We claim that 
	\begin{equation}\label{eqn43}
	x_N\notin A\Longrightarrow x_i\notin A,\qquad\forall 0\leq i\leq N.
	\end{equation}
	Clearly $x=x_0\notin A$. If $x_N\notin F_{w,\lambda}(\tilde{V}_{m-k-1,\lambda}\setminus V_0)$ and $x_i\in F_{w,\lambda}(\tilde{V}_{m-k-1,\lambda}\setminus V_0)$ for some $0<i<N$ and $w\in \{1,2,3\}^k\times \{4\}$, then there exist $0<i'<i$ and $i<i''<N$ such that 
	
	1). $x_{i'}\in F_{w,\lambda}V_0$ and $x_{i''}\in F_{w,\lambda}V_0$,
	
	2). for any $i'\leq j\leq i''$, we have $x_j\in F_{w,\lambda}(\tilde{V}_{m-k-1,\lambda})$.
	
	\noindent This is because only the three vertices $F_{w,\lambda}(p_i),i=1,2,3$ in $F_{w,\lambda}(\tilde{V}_{m-k-1,\lambda})$ are possibly connected to $\tilde{V}_{m,\lambda}\setminus F_{w,\lambda}(\tilde{V}_{m-k-1,\lambda})$ in the graph $G^{(m)}_\mathcal{J}$. However, 1),2) are not possible. Noticing that $x_{i'}\neq x_{i''}$ by the assumption that $x_0,x_1,\cdots,x_N$ is self-avoiding,  $y_0=F_{w,\lambda}^{-1}(x_{i'}), y_1=F_{w,\lambda}^{-1}(x_{i'+1}),\cdots, y_{i''-i'}=F_{w,\lambda}^{-1}(x_{i''})$ is a $G^{(m-k-1)}_{\mathcal{J}}$-path connecting two distinct vertices $y_0,y_{i''-i'}$ in $V_0$. By Definition \ref{def43} and Lemma \ref{lemma44}, this implies that $y_0\mathcal{J}y_{i''-i'}$, which contradicts the assumption of Case 2.
	
	Next, as a consequence of (\ref{eqn42}) and (\ref{eqn43}), we see that 
	\[x_N\notin A \Longrightarrow \{x_i\}_{i=0}^N\subset \bigcup_{w\in \{1,2,3\}^m}F_{w,\lambda}(\tilde{V}_{0,\lambda}).\]
	So, it suffices to consdier the restriction of the graph $G^{(m)}_\mathcal{J}$ to $\bigcup_{w\in \{1,2,3\}^m}F_{w,\lambda}(\tilde{V}_{0,\lambda})$, where each cell $F_{w,\lambda}(\tilde{V}_{0,\lambda}),w\in \{1,2,3\}^m$ is connected to the outside with at most three vertices $F_{w,\lambda}V_0$. By a similar argument as before, and noticing (\ref{eqn41}), we can see that
	\begin{equation}\label{eqn44}
	x_N\notin A\Longrightarrow \{x_i\}_{i=0}^N\subset \bigcup\big\{F_{w,\lambda}(\tilde{V}_{0,\lambda}):w\in \{1,2,3\}^m, x\in F_{w,\lambda}V_0\big\}.
	\end{equation}
	
	Back to the study of $\mathcal{J}$, let $y\in \tilde{V}_{0,\lambda}\setminus \{x\}$, we claim that $x{\mathcal{J}\mkern-10.5mu\backslash} y$. Otherwise, there exists a self-avoiding $G^{(m)}_\mathcal{J}$-path $x=x_0,x_1,\cdots,x_N=y$ by Definition \ref{def43} and Lemma \ref{lemma44}, which contradicts (\ref{eqn44}) since $y\notin A$ and $y\notin \bigcup\big\{F_{w,\lambda}(\tilde{V}_{0,\lambda}):w\in \{1,2,3\}^m, x\in F_{w,\lambda}V_0\big\}$ by (\ref{eqn41}). The argument applies to any $x,y\in \tilde{V}_{0,\lambda}$. So $\mathcal{J}=0$. 
\end{proof}

As a consequence, we get the existence of self-similar resistance forms on $\AG_\lambda$ for $\lambda\in \mathbb{D}$. 

\begin{theorem}\label{thm46}
	Let $\lambda\in \mathbb{D}\cap (0,\frac{1}{2})$ and $0<s<\infty$. There exists a unique $\frac{3}{5}\leq r<1$ and a unique $\mathcal{D}\in \mathcal{RF}(\tilde{V}_{0,\lambda})$ such that 
	\[
	\mcD=\Lambda_{\bm{r}}\mcD
	\]
	holds with $r_1=r_2=r_3=r$, $r_4=s$ and $\bm{r}=\{r_i\}_{i=1}^4$. In particular, Theorem \ref{thm36} holds for $\lambda\in \mathbb{D}\cap (0,\frac{1}{2})$. 
\end{theorem}
\begin{proof}
	Let $\tilde{\bm{r}}=\{\tilde{r}_i\}_{i=1}^4$ with $\tilde{r}_4\in (0,\infty)$ and $\tilde{r}_1=\tilde{r}_2=\tilde{r}_3=1$. Then, by Theorem \ref{thm42} and Proposition \ref{prop45}, there exists $C=C(\tilde{r}_4)>0$ and $\mcD\in \mathcal{RF}(\tilde{V}_{0,\lambda})$ such that $\Lambda_{\tilde{\bm{r}}}\mcD=C\mcD$. By setting $\bm{r}=\{r_i\}_{i=1}^4$ with $r_i=\tilde{r}_i\cdot C(\tilde{r}_4)$, $1\leq i\leq4$, we have $\Lambda_{\bm{r}}\mcD=\mcD$ holds. 
	
	We need to show $s=r_4=\tilde{r}_4\cdot C(\tilde{r}_4)$ can be any positive real number. This follows from the following two observations.\vspace{0.1cm}
	
	\noindent\textit{Observation 1.} $C(\tilde{r}_4)$ is a continuous and decreasing function of $\tilde{r}_4$. 
	
	Let $\tilde{\bm{r}}'=\{\tilde{r}'_i\}_{i=1}^4$, with $\tilde{r}'_1=\tilde{r}'_2=\tilde{r}'_3=1$ and $\tilde{r}'_4\leq \tilde{r}_4$.
	Then we have $\tilde{r}'_4\Lambda_{\tilde{\bm{r}}'}\mcD(f)\leq \tilde{r}_4\Lambda_{\tilde{\bm{r}}}\mcD(f)$ and $\Lambda_{\tilde{\bm{r}}}\mcD(f)\leq \Lambda_{\tilde{\bm{r}}'}\mcD(f)$. In addition, we have
	\begin{equation}\label{eqn46}
	\inf\frac{\Lambda_{\bm{\tilde{r}'}}\mcD(f)}{\mcD(f)}\leq C(\tilde{r}'_4)\leq \sup\frac{\Lambda_{\bm{\tilde{r}'}}\mcD(f)}{\mcD(f)},
	\end{equation}
	where the infimum and supremum are taken over $f\in l(\tilde{V}_{0,\lambda})\setminus Constants$. See \cite{M1,Sabot} for a proof of (\ref{eqn46}). It follows that $C(\tilde{r}_4)\leq C(\tilde{r}'_4)\leq C(\tilde{r}_4)\frac{\tilde{r}_4}{\tilde{r}'_4}$.
	
	\noindent\textit{Observation 2.} $\frac{3}{5}\leq C(\tilde{r}_4)<1$. 
	
	The fact $C(\tilde{r}_4)<1$ is a consequence of Proposition 3.1.8 of \cite{ki3}. For the remaining half, we let $\mcD'$ be the trace of the standard form $(\mcE,\mcF)$ on the Sierpinski gasket on $\tilde{V}_{0,\lambda}$, then we have 
	\[C(\tilde{r}_4)=\sup\frac{\Lambda_{\bm{\tilde{r}}}\mcD(f)}{\mcD(f)}\geq \inf\frac{\Lambda_{\bm{\tilde{r}}}\mcD'(f)}{\mcD'(f)}\geq \inf\frac{\Lambda_{\bm{\tilde{r}'}}\mcD'(f)}{\mcD'(f)}=\frac{3}{5},\]
	where $\bm{\tilde{r}'}=\{\tilde{r}'_i\}_{i=1}^4$ with $\tilde{r}'_1=\tilde{r}'_2=\tilde{r}'_2=1$ and $\tilde{r}'_4=+\infty$. See the book \cite{s} for the standard form on the standard Sierpinski gasket.
	
	It remains to show the uniqueness of $r$ for fixed $r_4=s$. Assume we have two solutions $r>r'$, and let $\bm{r}=(r,r,r,s)$, $\bm{r}'=(r',r',r',s)$, $\Lambda_{\bm{r}}\mcD=\mcD$ and $\Lambda_{\bm{r}'}\mcD'=\mcD'$. Then
	\[\frac{\mcD(f)}{\Lambda_{\bm{r'}}\mcD(f)}<\frac{\mcD(f)}{\Lambda_{\bm{r}}\mcD(f)}=1,\quad \forall f\in l(\tilde{V}_{0,\lambda})\setminus Constants.\]
	By taking the supremum over the compact set $\{f\in l(\tilde{V}_{0,\lambda}):\mcD(f)=1\}$, we get the contradiction $1=\inf\frac{\mcD'(f)}{\Lambda_{\bm{r'}}\mcD'(f)}\leq \sup\frac{\mcD(f)}{\Lambda_{\bm{r'}}\mcD(f)}<1$.  
\end{proof}

\noindent\textbf{Remark.} For $s\geq 1$ cases, we also have resistance forms on the closure of $V_{*,\lambda}$ with respect to the resistance metric, which contains $\bigcup_{w\in W_*}F_{w,\lambda}\partial\blacktriangle$.

\section{Resistance estimates}
In this section, we provide some esimtates. To see the common properties of $\AG_\lambda, \lambda\in \mathbb{D}\cap (0,\frac{1}{2})$, we return to the graphs $G_{m,\lambda}$.  We focus on $0<s<1$ cases, and some results (Proposition \ref{prop54} and \ref{prop56}) in this section extend to general cases $0<s<\infty$. 

First, we introduce some notations. 

\begin{definition}\label{def52} For $\lambda\in \mathbb{D}\cap (0,\frac{1}{2})$ and $0<s<1$, noticing that Theorem \ref{thm36} holds, we take the same notations $\frac{3}{5}\leq r(\lambda,s)<1$, $(\mathcal{E}_{\lambda,s},\mathcal{F}_{\lambda,s})\in \mathscr{E}(\lambda,s,r(\lambda,s))$ and $R_{\lambda,s}(x,y)$ as in Definition \ref{def37}. For $m\geq 1$, define $\mcE^{(m)}_{\lambda,s}=[\mcE_{\lambda,s}]_{V_{m,\lambda}}$, which has the form 
	\[\mcE^{(m)}_{\lambda,s}(f)=\sum_{\{x,y\}\in E_{m,\lambda}}c^{(m)}_{x,y}\big(f(x)-f(y)\big)^2,\]
	where $c^{(m)}_{x,y}\in [0,\infty)$ are constants depending on $\lambda,s$. We also use the notation $c^{(m)}_{x,y}(\lambda,s)=c^{(m)}_{x,y}$ to indicate its dependence on $\lambda,s$.
\end{definition}

In particular, we always have 
$\mcE^{(0)}_{\lambda,s}(f)=\frac{1}{2}\sum_{i\neq j}\big(f(p_i)-f(p_j)\big)^2$
by the choice of $(\mathcal{E}_{\lambda,s},\mathcal{F}_{\lambda,s})$. 

For convenience, in this section, we will write $r_1=r_2=r_3=r(\lambda,s)$, $r_4=s$ and $r_w=r_{w_1}r_{w_2}\cdots r_{w_m}$ for $w=w_1w_2\cdots w_m\in W_m,m\geq 0$, when $\lambda,s$ are fixed.

\subsection{Resistance estimates on $\partial\blacktriangle$}  We will frequently use the harmonic extension of a function on a set. We briefly explain our setting here:

Let $A=\bigcup_{w\in \Lambda}F_{w,\lambda}\AG_\lambda$, where $\Lambda\subset W_m,m\geq 0$. We consider $f\in C(B)$, where $B\subset A$, such that 
$\inf\{\sum_{w\in \Lambda}r_w^{-1}\mcE_{\lambda,s}(g\circ F_{w,\lambda}):g|_B=f,g\in C(A)\}<\infty$.  
We say $h\in C(A)$ is the harmonic extension of $f$ on $A$ if $h|_B=f$ and 
\[\sum_{w\in \Lambda}r_w^{-1}\mcE_{\lambda,s}(h\circ F_{w,\lambda})=\inf\{\sum_{w\in \Lambda}r_w^{-1}\mcE_{\lambda,s}(g\circ F_{w,\lambda}):g|_B=f,g\in C(A)\}.\]
The harmonic extension $h$ is unique.

Our first lemma concerns the effective resistance between a pair of disjoint sets $A,A'$: 
\[R_{\lambda,s}(A,A')=\sup\{\frac{|f(x)-f(y)|^2}{\mcE_{\lambda,s}(f)}:f\in \mcF_{\lambda,s}, f|_A=1,f|_{A'}=0\}. \] 

\begin{lemma}\label{lemma53}
	$R_{\lambda,s}(p_1,\partial_\downarrow\blacktriangle)\geq \frac{1}{2}\cdot\frac{s\cdot r(\lambda,s)}{s+r(\lambda,s)}$ for $\lambda\in \mathbb{D}\cap (0,\frac{1}{2})$ and $0<s<1$. 
\end{lemma}
\begin{proof}
	Define $u''\in C(F_{2,\lambda}\AG_\lambda\cup F_{3,\lambda}\AG_\lambda\cup\{p_1\})$ to be
	\[u''(x)=\begin{cases}
	1,\quad\text{ if }x=p_1,\\
	0,\quad\text{ if }x\in F_{2,\lambda}\AG_\lambda\cup F_{3,\lambda}\AG_\lambda.
	\end{cases}\]
	Then, let $u'$ be the harmonic extension of $u''$ on $\bigcup_{i=1}^3F_{i,\lambda}\AG_\lambda$. Finally, we take the harmonic extension of $u'$ on $\AG_\lambda$ to get $u\in \mathcal{F}_{\lambda,s}$. By the construction, 
	\[
	\mcE_{\lambda,s}(u\circ F_{i,\lambda})
	\begin{cases}
	=2,\quad\text{ if }i=1,\\
	=0,\quad\text{ if }i=2,3,\\
	\leq 2,\quad\text{ if }i=4.
	\end{cases}
	\]  
	As a consequence,
	$\mcE_{\lambda,s}(u)=\sum_{i=1}^4r_i^{-1}\mcE(u\circ F_{i,\lambda})\leq 2s^{-1}+2r(\lambda,s)^{-1}.$
\end{proof}

Using Lemma \ref{lemma53}, we can find a lower bound estimate for the resistance metric for $x,y\in \partial\blacktriangle$, and the upper bound estimate follows from a routine argument.

\begin{proposition}\label{prop54}
	There exist constants $0<c_1,c_2<\infty$ such that for any $\lambda\in \mathbb{D}\cap (0,\frac{1}{2})$, $0<s<1$ and $x,y\in \partial\blacktriangle$,  
	\[c_1\cdot s\cdot d(x,y)^{\theta(\lambda,s)}\leq R_{\lambda,s}\big(x,y\big)\leq c_2\cdot\frac{1}{1-r(\lambda,s)}\cdot d(x,y)^{\theta(\lambda,s)},\]
	where $\theta(\lambda,s)=-\frac{\log r(\lambda,s)}{\log 2}$.
\end{proposition}
\begin{proof}
	Let $x\neq y$ and $m=[-\frac{\log d(x,y)}{\log 2}]+2$.
	
	First, we show the upper bound. Let $\omega\in \pi_\lambda^{-1}(x)$. It is easy to see  
	\[R_{\lambda,s}(x,F_{[\omega]_m}p_{\omega_{m+1}})\leq \sum_{k=m}^\infty R_{\lambda,s}(F_{[\omega]_k}p_{\omega_{k+1}},F_{[\omega]_{k+1}}p_{\omega_{k+2}})\leq \frac{2}{3}\frac{r(\lambda,s)^{m+1}}{1-r(\lambda,s)}.\]
	The same estimate holds for $R_{\lambda,s}(y, F_{[\eta]_m}p_{\eta_{m+1}})$, where $\eta\in \pi^{-1}(y)$. In addition, there is a finite path $F_{[\omega]_m}p_{\omega_{m+1}}=x_0,x_1,\cdots,x_N=F_{[\eta]_m}p_{\eta_{m+1}}\in \bigcup_{w\in \Lambda_m}F_{w,\lambda}V_0$ for some $N\leq 6$. So 
	\[R_{\lambda,s}(x,y)\leq R_{\lambda,s}(x,x_0)+R_{\lambda,s}(x_0,x_N)+R_{\lambda,s}(x_N,y)\leq (4+\frac{4}{3}\frac{r(\lambda,s)}{1-r(\lambda,s)})r(\lambda,s)^m.\]
	The upper bound follows immediately by the choice of $m$.
	
	Next, we prove the lower bound estimate. Write $V_m=\bigcup_{w\in \{1,2,3\}^m}V_0$ for short. For each $z\in \partial\blacktriangle$, we define 
	\[U_m(z)=\bigcup\big\{F_w\AG_\lambda:w\in \{1,2,3\}^m, z\in F_w\AG_\lambda\big\}.\]
	Without loss of generality, we assume $x\in \partial_\downarrow\blacktriangle$, and define $u''_x\in C\big(V_m\bigcup (U_m(x)\bigcap \partial_\downarrow\blacktriangle)\big)$ by   
	\[u''_x(z)=\begin{cases}
	0,\quad\text{ if }z\in V_m\setminus(U_m(x)\cap \partial_\downarrow\blacktriangle),\\
	1,\quad\text{ if }z\in U_m(x)\cap \partial_\downarrow\blacktriangle.
	\end{cases}\] 
	Next, we let $u'_x$ be the harmonic extension of $u''_x$ on $\bigcup_{w\in \{1,2,3\}^m}F_{w,\lambda}\AG_\lambda$. By Lemma \ref{lemma53}, we have the upper bound
	\[
	\sum_{w\in \{1,2,3\}^m}r_w^{-1}\mcE_{\lambda,s}(u'_x\circ F_{w,\lambda})\leq 2(2s^{-1}+2r(\lambda,s)^{-1})r(\lambda,s)^{-m}+4r(\lambda,s)^{-m}.
	\]
	Finally, we let $u_x$ be the harmonic extension of $u'_x$ on $\AG_\lambda$. Noticing that $u'_x$ takes nonzero values on at most $4$ cells of the form $F_{w,\lambda}\AG_\lambda, w\in \{1,2,3\}^m$, $u_x$ has nonzero values on at most $12$ cells of the form $F_{w,\lambda}\AG_\lambda$, $w\in \bigcup_{k=0}^{m-1}\{1,2,3\}^k\times \{4\}$. In addition, by the Markov property, we have $0\leq u_x\leq 1$, so
	\begin{equation}\label{eqn52}
	\mathcal{E}_{\lambda,s}(u_x)\leq 36s^{-1}r(\lambda,s)^{-m+1}+\sum_{w\in \{1,2,3\}^m}r_w^{-1}\mcE_{\lambda,s}(u'_x\circ F_{w,\lambda})\leq c_3\cdot s^{-1}\cdot r(\lambda,s)^{-m}.
	\end{equation}
	Lastly, noticing that $U_m(x)\cap U_m(y)=\emptyset$ by the choice of $m$, we have $u_x(x)=1$ and $u_x(y)=0$. The lower bound estimate follows immediately from the estimate (\ref{eqn52}).
\end{proof}

\subsection{A bound for $r(\lambda,s)$} In Proposition \ref{prop54}, we have a term $\big(1-r(\lambda,s)\big)^{-1}$ in the upper bound. As an improvement, we will show that $r(\lambda,s)$ is uniformly bounded away from $1$ when we have the restrictions $\lambda\in [a,\frac{1}{2}-a]$ and $s\geq b$ for some $0<a<\frac{1}{4}$ and $b>0$.

\begin{lemma}\label{lemma55}
	Let $\lambda\in \mathbb{D}\cap(0,\frac{1}{2})$ and $0<s<1$. Let $q_{1,\lambda}$ be the unique vertice on $\partial_\downarrow\blacktriangle$ such that $F_{1,\lambda}(q_{1,\lambda})=F_{4,\lambda}(p_1)$, and define $h'_{\lambda,s}$ on $\{p_1,p_2,p_3,q_{1,\lambda}\}$ by \[h'_{\lambda,s}(p_1)=0,\quad h'_{\lambda,s}(p_2)=h'_{\lambda,s}(p_3)=h'_{\lambda,s}(q_{1,\lambda})=1.\]
	Let $h_{\lambda,s}$ be the harmonic extension of $h'_{\lambda,s}$ on $\AG_\lambda$. Then, \[r(\lambda,s)=\frac{c^{(1)}_{p_1,F_{1,\lambda}p_2}h_{\lambda,s}(F_{1,\lambda}p_2)+c^{(1)}_{p_1,F_{1,\lambda}p_3}h_{\lambda,s}(F_{1,\lambda}p_3)+c^{(1)}_{p_1,F_{4,\lambda}p_1}h_{\lambda,s}(F_{4,\lambda}p_1)}{c^{(1)}_{p_1,F_{1,\lambda}p_2}+c^{(1)}_{p_1,F_{1,\lambda}p_3}+c^{(1)}_{p_1,F_{4,\lambda}p_1}}.\]
\end{lemma}
\begin{proof}
	Let $V=\{p_1,p_2,p_3,q_{1,\lambda}\}$ and define $c_{x,y}$ for $x\neq y\in V$ with
	\[[\mcE_{\lambda,s}]_{V}(f)=\frac{1}{2}\sum_{x\neq y\in V}c_{x,y}\big(f(x)-f(y)\big)^2.\]
	Then, by the self-similarity of $(\mathcal{E}_{\lambda,s},\mathcal{F}_{\lambda,s})$, we can see that 
	\begin{equation}\label{eqn53}
	c^{(1)}_{F_{1,\lambda}x,F_{1,\lambda}y}=r(\lambda,s)^{-1}c_{x,y}
	\end{equation}
	for any $x\neq y\in V$. By the definition of $h_{\lambda,s}$, we have 
	\begin{equation}\label{eqn54}
	\begin{aligned}
	\sum_{x\in V\setminus \{p_1\}}c_{x,p_1}h(x)&=[\mcE_{\lambda,s}]_{V}(1-h_{\lambda,s})\\
	&=[\mcE_{\lambda,s}]_{V\cup V_{1,\lambda}}(1-h_{\lambda,s})=\sum_{x\in V\setminus \{p_1\}} c^{(1)}_{F_{1,\lambda}x,p_1}h(F_{1,\lambda}x).
	\end{aligned}
	\end{equation}
	The Lemma follows from equations (\ref{eqn53}) and (\ref{eqn54}). 
\end{proof}

\begin{proposition}\label{prop56}
	Let $0<a<\frac{1}{4}$ and $0<b<1$. There is a constant $c<1$ such that  
	\[r(\lambda,s)\leq c,\qquad \forall \lambda\in \mathbb{D}\cap [a,\frac{1}{2}-a],\text{ }s\in [b,1).\]
\end{proposition}
\begin{proof}
	Choose $m\geq 0$ so that $2^{-m-1}<a$. We consider the set 
	\[A=\bigcup_{w\in W_A}F_{w,\lambda}\AG_\lambda,\text{ with } W_A=\{1,4,21^{(m)},12^{(m)}\}.\]
	Also, we set 
	\[\tilde{V}_A=\{F_{1,\lambda}p_2,F_{1,\lambda}p_3,F_{4,\lambda}p_1,p_1\},\text{ and }V_A=\tilde{V}_A\cup\{p\},\] 
	where $p$ represents the lower boundary $B=\bigcup_{w\in W_A\setminus \{1\}}F_{w,\lambda}\partial_\downarrow\blacktriangle$. See Figure \ref{fig41} for an illustration.
	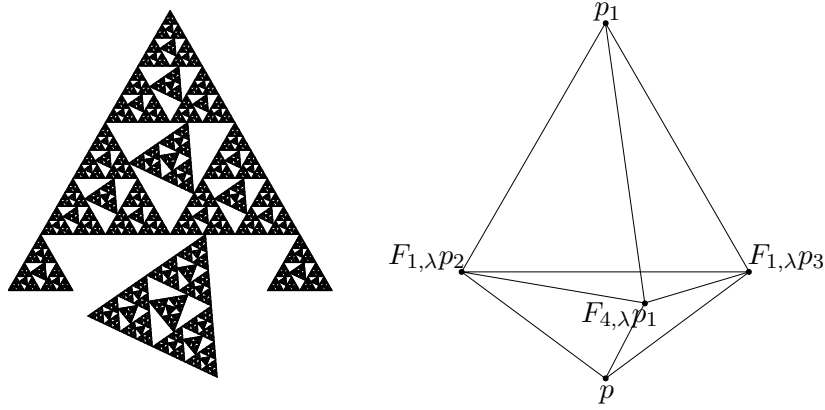
\begin{figure}[htp]
		\includegraphics[width=4.5cm]{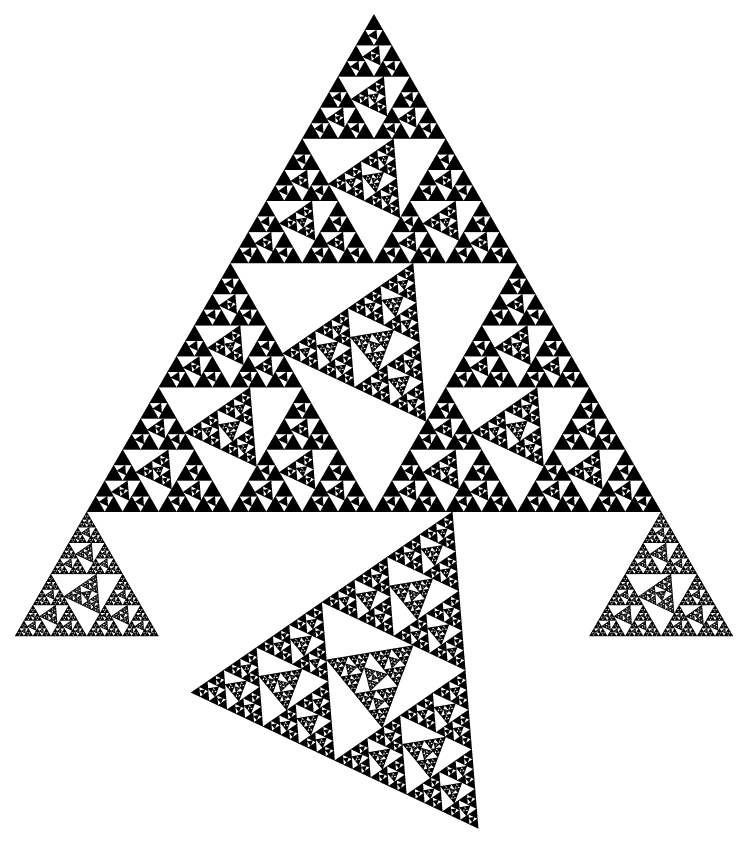}\qquad\qquad
		\includegraphics[width=4cm]{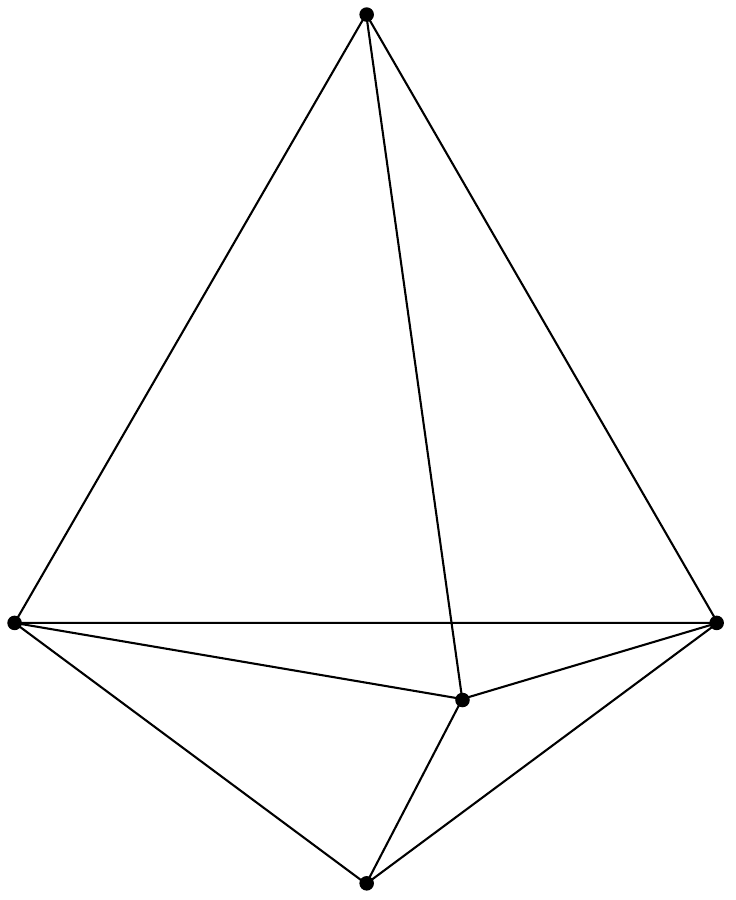}
		\begin{picture}(0,0)
		\put(-63,-5){$p$}
		\put(-65,140){$p_{1}$}
		\put(-143,46){$F_{1,\lambda}p_2$}
		\put(-7,46){$F_{1,\lambda}p_3$}
		\put(-70,24){$F_{4,\lambda}p_1$}
		\end{picture}
		\caption{The area $A$ and the set $V_A$}\label{fig41}
	\end{figure} 
	
	Let $\tilde{h}'=1_B\in C(\{p_1\}\cup B)$ and define $\tilde{h}$ to be the harmonic extension of $\tilde{h}'$ on $A$. It is then clear that $\tilde{h}(F_{1,\lambda}x)\geq h_{\lambda,s}(F_{1,\lambda}x)$ for each $x\in \tilde{V}_A$, where $h_{\lambda,s}$ is the same function as in Lemma \ref{lemma55}. Thus, by Lemma \ref{lemma55}, we can see that 
	\begin{equation}\label{eqn55}
	r(\lambda,s)\leq \max_{x\in \tilde{V}_A}h_{\lambda,s}(x)\leq \max_{x\in\tilde{V}_A}\tilde{h}(x).
	\end{equation}
	To simplify the question, we let $p$ represent $B$ and study the following energy on $V_A$,
	\[\mcE_{V_A}(f)=\frac{1}{2}\sum_{x\neq y\in \tilde{V}_A}c^{(1)}_{x,y}\big(f(x)-f(y)\big)^2+\sum_{w\in W_A\setminus \{1\}}r_w^{-1}R_{\lambda,s}(p_1,\partial_\downarrow\blacktriangle)^{-1}\big(f(p)-f(F_{w,\lambda}p_1)\big)^2.\]
	In a natural way, with a slight abuse of the notation, we view $\tilde{h}$ as a function on $V_A$.  	
	
	We consider a class $\mathscr{D}_{V_A}\subset \mathcal{RF}(V_A)$, where each $\mcD\in \mathscr{D}_{V_A}$ has the form
	\[\mcD(f)=\frac{1}{2}\sum_{x\neq y\in \tilde{V}_A}c_{x,y}\big(f(x)-f(y)\big)^2+\sum_{x\in \tilde{V}_A\setminus \{p_1\}} c_{p,x}\big(f(p)-f(x)\big)^2,\]
	and satisfies the following requirements 1-4.
	
	1. Let $\mcD_{\tilde{V}_A}(f)=\frac{1}{2}\sum_{x\neq y\in \tilde{V}_A}c_{x,y}\big(f(x)-f(y)\big)^2$. Then there exists $1\leq c_1\leq \frac{5}{4}$ such that
	\[[\mcD_{\tilde{V}_A}]_{F_{1,\lambda}V_0}(f)=\frac{c_1}{2}\sum_{i\neq j}\big(f(F_{1,\lambda}p_i)-f(F_{1,\lambda}p_j)\big)^2.\] 
	
	2. $c_{p,x}\leq 2(\frac{5}{4})^{m+1}(\frac{5}{4}+b^{-1})$, for $x\in \{F_{1,\lambda}p_2,F_{1,\lambda}p_3\}$; $c_{p,x}\leq \frac{5}{2b}(\frac{5}{4}+b^{-1})$, for $x=F_{4,\lambda}p_1$.
	
	3. $c_{F_{4,\lambda}p_1,x}\leq c^{-1}_{5.4.1}\frac{5}{4}b^{-1}a^{\frac{\log 4-\log 5}{\log 2}}$, for $x\in F_{1,\lambda}V_0$, where $c_{5.4.1}$ is the constant in the lower bound estimate of Proposition \ref{prop54}.
	
	4. $\max_{x\in V_0}c_{F_{4,\lambda}p_1, F_{1,\lambda}x}\geq \frac{1}{18}$.
	
	For each $\mcD\in \mathscr{D}_{V_A}$, we define $h_{\mcD}$ as the unique function on $l(V_A)$ such that 
	\[\begin{cases}
	h_{\mcD}(p)=1,\quad h_{\mcD}(p_1)=0,\\
	h_{\mcD}\text{ is harmonic on }V_A\setminus\{p,p_1\}. 
	\end{cases}\]
	As a consequence of requirements 1,2,4, we have $\max_{x\in \tilde{V}_A} h_{\mcD}(x)<1$ for any $\mcD\in \mathscr{D}_{V_A}$.  Since the requirements 1-4 provide a uniform upper bound for all the conductances, $\mathscr{D}_{V_A}$ is compact with the natural topology. Thus $\sup_{\mcD\in \mathscr{D}_{V_A}}\max_{x\in \tilde{V}_A} h_{\mcD}(x)<1$.

	One can check that if $r(\lambda,s)\geq \frac{4}{5}$, $\mcE_{V_A}\in \mathscr{D}_{V_A}$: the requirement 1 is natural by the self-similarity of $(\mcE_{\lambda,s},\mcF_{\lambda,s})$; the requirement 2 is a consequence of Lemma \ref{lemma53}; the requirement 3 is a consequence of the lower bound estimate in Proposition \ref{prop54}; the requirement 4 can be proven by contradiction, if $r(\lambda,s)\geq\frac{4}{5}$ and $\max_{x\in V_0}c^{(1)}_{F_{4,\lambda}p_1, F_{1,\lambda}x}<\frac{1}{18}$, then one can easily check $[\mcE^{(1)}_{\lambda,s}]_{V_0}(f)<\frac{1}{2}\sum_{i\neq j}\big(f(p_i)-f(p_j)\big)^2=\mcE^{(0)}_{\lambda,s}(f)$ for any $f\in l(V_0)\setminus Constants$. Combining the above observations with equation (\ref{eqn55}), we finally have 
	\[r(\lambda,s)\leq \max\big\{\frac{4}{5},\sup_{\mcD\in \mathscr{D}_A}\max_{x\in \tilde{V}_A}h_{\mcD}(x)\big\}<1.\]
	It is easy to see that the above estimate is independent of $a\leq\lambda\leq\frac{1}{2}-a$, $s\geq b$. 
\end{proof}

\subsection{Resistance estimates on $\AG_\lambda$}
At the end of this section, we provie a rough resistance estimate on $\AG_\lambda$ for completeness. In particular, the estimate shows the family $R_{\lambda,s}$ satisfies the condition of Theorem \ref{thm29}.

\begin{proposition}\label{prop57}
	Let $\lambda\in \mathbb{D}\cap [a,\frac{1}{2}-a]$, $s\in [b,1-b]$ for some $0<a<\frac{1}{4}$ and $0<b<\frac{1}{2}$. Let $\eta_*(\lambda,s)=\min\{\frac{\log s}{\log \rho},-\frac{\log r(\lambda,s)}{\log 2}\}$ and $\eta^*(\lambda,s)=\max\{\frac{\log s}{\log \rho},-\frac{\log r(\lambda,s)}{\log 2}\}$, where $\rho$ is the contraction ratio of $F_{4,\lambda}$. Then there are constants $0<c_1,c_2<\infty$ depending only on $a,b$ such that
	\[c_1d(x,y)^{\eta^*(\lambda,s)}\leq R_{\lambda,s}(x,y)\leq c_2d(x,y)^{\eta_*(\lambda,s)},\qquad\forall x,y\in \AG_\lambda.\]
\end{proposition}
\begin{proof}
	The upper bound is proven with a same routine argument as in Proposition \ref{prop54}, where we need Proposition \ref{prop56} and the fact $s\leq 1-b$ to get $c_2<\infty$.
	
	For the lower bound, we choose $m\geq 0, w\in W_m$ such that $x\in F_{w,\lambda}\AG_\lambda$ and $\diam(F_{w,\lambda}\AG_\lambda)=\sup_{z,z'\in F_{w,\lambda}\AG_\lambda}d(z,z')\in [8^{-1}d(x,y),2^{-1}d(x,y)]$. Since $\lambda\in \mathbb{D}\cap [a,\frac{1}{2}-a]$, we can find $m_0>0$ depending only on $a$ such that 
	\[F_{\tau,\lambda}\AG_\lambda\cap F_{\tau',\lambda}\AG_\lambda=\emptyset, \]
	for any $\tau\neq\tau'\in W_{m+m_0}$ satisfying $F_{\tau,\lambda}\AG_\lambda\cap F_{w,\lambda}\AG_\lambda\neq \emptyset$, $F_{\tau',\lambda}\AG_\lambda\cap F_{w,\lambda}\AG_\lambda\neq \emptyset$. Define 
	\[U=\bigcup\big\{F_{\tau,\lambda}\AG_\lambda:\tau\in W_{m+m_0},F_{\tau,\lambda}\AG_\lambda\cap F_{w,\lambda}\AG_\lambda\neq \emptyset\big\}.\]
	Let $u'\in C\big(F_{w,\lambda}\AG_\lambda\bigcup (\AG_\lambda\setminus U)\big)$ be defined by $u'|_{F_{w,\lambda}\AG_\lambda}=1$, $u'|_{\AG_\lambda\setminus U_m(x)}=0$, and let $u$ be the harmonic extension of $u'$ on $\AG_\lambda$. Then $u(x)=1$, $u(y)=0$ and $\mcE_{\lambda,s}(u)\leq cd(x,y)^{-\eta^*(\lambda,s)}$ for some $0<c<\infty$ depending on $a,b$ by applying Lemma \ref{lemma53}.  
\end{proof}

\section{Existence and uniqueness}
In this section, we prove Theorem \ref{thm36}. The existence of a self-similar resistance form follows easily from Theorem \ref{thm29}, \ref{thm213} and Proposition \ref{prop56},\ref{prop57}. On the other hand, the proof of uniqueness needs a little more care, based on Lemma \ref{lemma27} and Proposition \ref{prop54}. 

\subsection{Proof of the existence}
Let $\lambda\in (0,\frac{1}{2})$ and $s\in (0,1)$. Choose $\{\lambda_n\}_{n\geq 1}\subset \mathbb{D}\cap (0,\frac{1}{2}),\{s_n\}_{n\geq 1}\subset (0,1)$ so that $\lambda_n\to \lambda,s_n\to s$ as $n\to \infty$. By Theorem \ref{thm29} and Proposition \ref{prop57}, there is a subsequence $\{n_k\}_{k\to\infty}$ such that $R_{\lambda_{n_k},s_{n_k}}\rightarrowtail R\in \mathcal{RM}(\AG_\lambda)$, and we can assume without loss of generality $r(\lambda_{n_k},s_{n_k})\to r$ for some $r\in [\frac{3}{5},1)$, ganranteed by Proposition \ref{prop56}. 

Let $(\mcE,\mcF)$ be the resistance form associated with $R$ on $\AG_\lambda$, we need to show that $(\mcE,\mcF)$ is self-similar. By Theorem \ref{thm213}, we know that $(\mcE_{\lambda_{n_k},s_{n_k}},\mcF_{\lambda_{n_k},s_{n_k}})$ $\Gamma$-converges to $(\mcE,\mcF)$ on $C(\blacktriangle)$. As a consequence, for each $f\in \mcF$, we can find $f_k \in \mcF_{\lambda_{n_k},s_{n_k}},k\geq 1$ so that $f_k\rightarrowtail f$ and 
\[\mcE(f)=\lim\limits_{k\to\infty}\mcE_{\lambda_{n_k},s_{n_k}}(f_k).\]
For $1\leq i\leq 4$, we also have $\mcE(f\circ F_{i,\lambda})\leq \liminf\limits_{k\to\infty}\mcE_{\lambda_{n_k},s_{n_k}}(f_k\circ F_{i,\lambda_{n_k}})$, noticing that $f_k\circ F_{i,\lambda_{n_k}}\rightarrowtail f\circ F_{i,\lambda}$. As a consequence, and by the self-similarity of $(\mcE_{\lambda_{n_k},s_{n_k}},\mcF_{\lambda_{n_k},s_{n_k}})$, we have
\[
\begin{aligned}
\mcE(f)=\lim\limits_{k\to\infty}\mcE_{\lambda_{n_k},s_{n_k}}(f_k)&\geq r^{-1}\sum_{i=1}^3\liminf\limits_{k\to\infty}\mcE_{\lambda_{n_k},s_{n_k}}(f_k\circ F_{i,\lambda_{n_k}})+s^{-1}\liminf\limits_{k\to\infty}\mcE_{\lambda_{n_k},s_{n_k}}(f_k\circ F_{4,\lambda_{n_k}})\\
&\geq r^{-1}\sum_{i=1}^3\mcE(f\circ F_{i,\lambda})+s^{-1}\mcE(f\circ F_{4,\lambda}).
\end{aligned}
\]
For the reverse inequality, we apply the remark below Theorem \ref{thm213}. For any $f\in \mcF$, we can find $f_k\in \mcF_{\lambda_{n_k},s_{n_k}},k\geq 1$ so that $f_k\rightarrowtail f$, $f_k(F_{i,\lambda_{n_k}}p_j)=f(F_{i,\lambda}p_j)$ for $1\leq i\leq 4,1\leq j\leq3$ and 
\[\mcE(f\circ F_{i,\lambda})=\lim\limits_{k\to\infty}\mcE_{\lambda_{n_k},s_{n_k}}(f_k\circ F_{i,\lambda_{n_k}}),\quad\forall 1\leq i\leq 4.\]
As a consequence, 
\[
\begin{aligned}
\mcE(f)\leq\liminf\limits_{k\to\infty}\mcE_{\lambda_{n_k},s_{n_k}}(f_k)&\leq r^{-1}\sum_{i=1}^3\lim\limits_{k\to\infty}\mcE_{\lambda_{n_k},s_{n_k}}(f_k\circ F_{i,\lambda_{n_k}})+s^{-1}\lim\limits_{k\to\infty}\mcE_{\lambda_{n_k},s_{n_k}}(f_k\circ F_{4,\lambda_{n_k}})\\
&=r^{-1}\sum_{i=1}^3\mcE(f\circ F_{i,\lambda})+s^{-1}\mcE(f\circ F_{4,\lambda}).
\end{aligned}
\]
This finishes the proof.

\subsection{Proof of the uniqueness}
The proof of the uniqueness of self-similar resistance forms is inspired by the celebrated work \cite{BBKT}. The following lemma plays a key role in the proof.

\begin{proposition}[\cite{BBKT}]
	Suppose $(\mcE_1,\mcF)$ and $(\mcE_2,\mcF)$ are local regular resistance forms on $X$ and that 
	\[\mcE_1(f)\leq \mcE_2(f),\qquad\forall f\in \mcF.\]
	Let $\delta>0$ and $\mcE=(1+\delta)\mcE_2-\mcE_1$. Then $(\mcE,\mcF)$ is also a local regular resistance form on $X$. 
\end{proposition}

\noindent\textbf{Remark.} The orginal theorem (Theorem 2.1 in \cite{BBKT}) is about local regualr irreducible conservative Dirichlet forms, and the proof applies to resistance forms without difficulty. \vspace{0.15cm}

The above proposition applies easily to self-similar resistance forms.
\begin{lemma}\label{lemma67}
	Let $0<r,s<1$, $0<\lambda<\frac{1}{2}$ and $(\mcE_i,\mcF)\in \mathscr{E}(\lambda,s,r)$ for $i=1,2$. Assume that $\mcE_1(f)\leq \mcE_2(f)$, $\forall f\in \mcF$. Define $\mcE'_\delta=(1+\delta)\mcE_2-\mcE_1$ for $\delta>0$. Then $(\mcE'_\delta,\mcF)\in  \mathscr{E}(\lambda,s,r), \forall \delta>0$.  
\end{lemma}

Next, observe that all the estimates in Section 5 can be applied to $\lambda\notin \mathbb{D}$ cases. In particular, Proposition \ref{prop54} holds. We can easily deduce the following result.

\begin{lemma}\label{lemma68}
	Let $(\mathcal{E}_i,\mathcal{F}_i)\in \mathscr{E}(\lambda,s,r),i=1,2$, then we have $\mathcal{F}_1=\mathcal{F}_2$. In addition, define 
	\[\sup(\mcE_1|\mcE_2)=\sup\{\frac{\mathcal{E}_1(f)}{\mathcal{E}_2(f)}:f\in \mcF_1\setminus Constants\},\quad \inf(\mcE_1|\mcE_2)=\inf\{\frac{\mathcal{E}_1(f)}{\mathcal{E}_2(f)}:f\in \mcF_1\setminus Constants\}.\] 
	Then, there exists a constant $1\leq c<\infty$ depending only on the coeffieicents $\lambda,s,r$ such that 
	\[\frac{\sup(\mcE_1|\mcE_2)}{\inf(\mcE_1|\mcE_2)}\leq c.\]
\end{lemma}
\begin{proof}
	Let $R_i(x,y)=\inf\{\frac{|f(x)-f(y)|^2}{\mcE_i(f)}:f\in \mcF_i\setminus Constants\}$. Without loss of generality, we assume that $[\mathcal{E}_i]_{V_0}(f)=\frac{1}{2}\sum_{j\neq k}\big(f(p_j)-f(p_k)\big)^2$ for $i=1,2$ and $f\in l(V_0)$. 
	
	For convenience, we write $V_{m,\lambda,w}=F_{w,\lambda}^{-1}(V_{m,\lambda})\cap \AG_\lambda$ for $m\geq 0$ and $w\in W_m$. Then, by the self-similarity of $(\mcE_i,\mcF_i),i=1,2$, we have
	\begin{equation}\label{eqn64}
	[\mcE_i]_{V_{m,\lambda}}(f)=\sum_{w\in W_m}r_w^{-1}[\mcE_i]_{V_{m,\lambda,w}}(f\circ F_{w,\lambda}), \quad\forall f\in l(V_{m,\lambda}),i=1,2,
	\end{equation}
	where $r_1=r_2=r_3=r, r_4=s$ and $r_w=r_{w_1}r_{w_2}\cdots r_{w_m}$ for $w=w_1w_2\cdots w_m\in W_m$. By  Proposition \ref{prop54}, there is a constant $c_1>0$ depending only on $\lambda,s$ such that 
	\[c_1^{-1}R_1(x,y)\leq R_2(x,y)\leq c_1R_1(x,y),\qquad\forall x\neq y\in \partial\blacktriangle.\]
	Thus, noticing that $V_{m,\lambda,w}\subset \partial\blacktriangle$ and $\#V_{m,\lambda,w}\leq 6$, by Lemma \ref{lemma27}, for any $m\geq 0, w\in W_m$ and $f\in l(V_{m,\lambda})$,
	\[(15c_1)^{-1}[\mcE_2]_{V_{m,\lambda,w}}(f\circ F_{w,\lambda})\leq [\mcE_1]_{V_{m,\lambda,w}}(f\circ F_{w,\lambda})\leq 15c_1\cdot[\mcE_2]_{V_{m,\lambda,w}}(f\circ F_{w,\lambda}).\]
	By taking the summation in (\ref{eqn64}), and then taking the limit, we have $\mcF_1=\mcF_2$ and
	\[(15c_1)^{-1}\cdot\mcE_2(f)\leq\mcE_1(f)\leq 15c_1\cdot\mcE_2(f),\qquad\forall f\in \mcF_1,\]
	by Theorem 2.3.7 of \cite{ki3}. It suffices to take $c=(15c_1)^2$. 
\end{proof}

Now we are ready to finish the proof of Theorem \ref{thm36}. 

\begin{proof}[Proof of Theorem \ref{thm36}]
	Existence is shown in Section 6.1. It remains to prove the uniqueness.
	
	First, we assume that there are $0<r_1<r_2<1$ such that $\mathscr{E}(\lambda,s,r_i)\neq \emptyset,i=1,2$. We choose $(\mcE_i,\mcF_i)\in \mathscr{E}(\lambda,s,r_i)$ such that 
	\begin{equation}\label{eqn65}
	[\mcE_i]_{V_0}(u)=\sum_{j\neq k}\big(u(p_j)-u(p_k)\big)^2,\quad\forall u\in l(V_0), i=1,2.
	\end{equation}
	We take the same notations as in the proof of Lemma \ref{lemma68}. First, we can easily see that $\mcF_1\subset \mcF_2$. In fact, by applying Proposition \ref{prop54} and a routine argument as before, we see that 
	\[R_1(x,y)\leq c\cdot R_2(x,y), \quad\forall x,y\in \partial\blacktriangle, \]
	for some constant $0<c<\infty$. So that 
	\[
	[\mcE_2]_{V_{m,\lambda}}(f)\leq 15c[\mcE_1]_{V_{m,\lambda}}(f),\quad \forall f\in l(V_{m,\lambda}),
	\]
	by equation (\ref{eqn64}) and Lemma \ref{lemma27} as in the proof Lemma \ref{lemma68}. The claim $\mcF_1\subset \mcF_2$ follows from a limit argument and Theorem 2.3.7 of \cite{ki3}. Next, we fix $u\in l(V_0)\setminus Constants$, and take $h_i$ to be the harmonic extension of $u$ on $\AG_\lambda$ with respect to $(\mcE_i,\mcF_i)$ for $i=1,2$. For a large $m$, we can see that 
	\begin{equation}\label{eqn67}
	r_2^{-m}\sum_{w\in \{1,2,3\}^m}\mcE_2(h_1\circ F_w)<r_1^{-m}\sum_{w\in \{1,2,3\}^m}\mcE_1(h_1\circ F_w).
	\end{equation}
	Let $h'=h_1|_{\bigcup_{w\in \{1,2,3\}^m}F_w\AG_\lambda}$ and let $h$ be the harmonic extension of $h'$ on $\AG_\lambda$ with respect to  $(\mcE_2,\mcF_2)$. Then by (\ref{eqn65}) and (\ref{eqn67}), we have 
	\[[\mcE_2]_{V_0}(u)=\mcE_2(h_2)\leq \mcE_2(h)<\mcE_1(h_1)=[\mcE_1]_{V_0}(u),\]
	which contradicts(\ref{eqn65}). Thus, we have proved the uniqueness of $r$. 
	
	The uniqueness (up to a constant multiplier) of $(\mcE,\mcF)\in \mathscr{E}(r,s,\lambda)$ follows from a same argument as \cite{BBKT}.  Assume there are different $(\mcE_1,\mcF_1)$ and $(\mcE_2,\mcF_2)$ in $\mathscr{E}(r,s,\lambda)$, then by Lemma \ref{lemma68}, we have $\mcF_1=\mcF_2$. Without loss of generality, we assume
	\[\sup(\mcE_1|\mcE_2)>\inf(\mcE_1|\mcE_2)=1.\]
	Then, according to Lemma \ref{lemma67}, by defining $\mcE'_\delta(f)=(1+\delta)\mcE_1(f)-\mcE_2(f)$, we have another form $(\mcE'_\delta,\mcF)$ in $\mathscr{E}(\lambda,s,r)$. Clearly, we have 
	\[
	\begin{cases}
	\sup(\mcE'_\delta|\mcE_2)=(1+\delta)\sup(\mcE_1|\mcE_2)-1,\\ \inf(\mcE'_\delta|\mcE_2)=(1+\delta)\inf(\mcE_1|\mcE_2)-1=\delta.
	\end{cases}
	\]
	So we have 
	\[\frac{\sup(\mcE'_\delta|\mcE_2)}{\inf(\mcE'_\delta|\mcE_2)}=\sup(\mcE_1|\mcE_2)+\delta^{-1}\big(\sup(\mcE_1|\mcE_2)-1\big).\]
	As $\delta$ can be arbitrarily small, this contradicts Lemma \ref{lemma68}.
\end{proof}

Finally, we finish the proof of Theorem \ref{thm38}.

\begin{proof}[Proof of Theorem \ref{thm38}]
	Let $\lambda_n\to\lambda\in (0,\frac{1}{2})$ and $s_n\to s \in (0,1)$. By Theorem \ref{thm213}, it suffices to show that $R_{\lambda_n,s_n}\rightarrowtail R_{\lambda,s}$. Assume not, there are $\varepsilon>0$, a subsequence $\{n_k\}_{k\geq 1}$ and $x_k\in \AG_{\lambda_{n_k}},y_k\in \AG_{\lambda_{n_k}}$, so that $x_k\to x,y_k\to y$ and 
	\[|R_{\lambda,s}(x,y)-R_{\lambda_{n_k},s_{n_k}}(x_k,y_k)|\geq \varepsilon,\quad\forall k\geq 1. \]
	On the other hand, by the proof in Section 6.1 and the uniqueness statement of Theorem \ref{thm36}, there is a further subsequence $\{n_{k_l}\}_{l\geq 1}$ so that $R_{\lambda_{n_{k_l}},s_{n_{k_l}}}\rightarrowtail R_{\lambda,s}$. A contradiction. 
\end{proof}

\noindent\textbf{Remark.} The fact that $(\mcE_{\lambda_n,s_n},\mcF_{\lambda_n,s_n})$ $\Gamma$-converges to $(\mcE_{\lambda,s},\mcF_{\lambda,s})$ on $C(\blacktriangle)$ can be proven directly with Theorem 8.3 of the book \cite{D}, noticing that $C(\blacktriangle)$ is a metric space.

\section{About irregular cases}
We may lossen the conditions (C1) in Section 2. \vspace{0.1cm}

\noindent (C1'). $(B,d)$ is a locally compact seperable metric space, $\{A_n\}_{n\geq 1}$ is a sequence of Borel sets in $B$ and $A$ is a Borel set in $B$. For each $m$, there is compact $A_{n,m}\subset A_n$ for each $n\geq 1$, and there is $A_{\infty,m}\subset A$, so that $\delta(A_{n,m},A_{\infty,m})\to 0$ as $n\to\infty$. In addition, $A_{n}=\bigcup_{m=1}^\infty A_{n,m}$ for each $n\geq 1$ and $A=\bigcup_{m=1}^\infty A_{\infty,m}$. \vspace{0.15cm}

With condition (C1'), for $f_n\in l(A_n),n\geq 1$ and $f\in l(A)$ such that $f_n|_{A_{n,m}}\in C(A_{n,m})$ and $f|_{A_{\infty,m}}\in C(A_{\infty,m})$, $\forall m\geq 1$,  we write $f_n\stackrel{,}{\rightarrowtail}f$ if $f_n|_{A_{n,m}}\rightarrowtail f|_{A_{\infty,m}}$ for all $m\geq 1$. The definition of $\stackrel{,}{\rightarrowtail}$ may depend on the choice of $A_{n,m}$. Also, $f_n\stackrel{,}{\rightarrowtail}f$ does not necessarily imply that $f_n\rightarrowtail f$, where $f_n\rightarrowtail f$ still means that $f_n(x_n)\to f(x)$ for any $x_n\to x$.

The example that we are most interested in this paper is of course still $\AG_\lambda,\lambda\in (0,\frac{1}{2})$. Let $\lambda_n\to \lambda\in (0,\frac{1}{2})$, $A_n=\bigcup_{w\in W_*} F_{w,\lambda_n}\partial\blacktriangle$, $A=\bigcup_{w\in W_*} F_{w,\lambda}\partial\blacktriangle$, $A_{n,m}=\bigcup_{w\in W_m} F_{w,\lambda_n}\partial\blacktriangle$ and $A_{\infty,m}=\bigcup_{w\in W_m} F_{w,\lambda}\partial\blacktriangle$. Then, we get a particular example satisfying (C1'). 

The main result in Section 2.3 still holds without difficulty. In particular, we have the following result.

\begin{theorem}\label{thm71}
	Assume (C1'). Let $R_n\in\mathcal{RM}(A_n)$ for $n\geq 1$ and $R\in \mathcal{M}(A)$. Assume that $R_n\in C(A_{n,m}\times A_{n,m}),\forall n,m\geq1$ and $R_n\stackrel{,}{\rightarrowtail} R$, then we have and $R\in C(A_{\infty,m}\times A_{\infty,m}),\forall m\geq 1$, and $R\in \mathcal{RM}(A)$. As a consequence, $(A,R)$ is separable.
	
	Let $(\mcE_n,\mcF_n)$ be the resistance form associated with $R_n$ for $n\geq 1$, and let $(\mcE,\mcF)$ be the resistance form associated with $R$. Then, we have $(\mcE_n,\mcF_n)$ converges to $(\mcE,\mcF)$ in the following sense:
	
    (a''). If $f_n\stackrel{,}{\rightarrowtail} f$, where $f_n\in l(A_n),\forall n\geq 1$ and $f_n|_{A_{n,m}}\in C(A_{n,m}),\forall n,m\geq 1$, then
    \[\mcE(f)\leq \liminf_{n\to\infty} \mcE_n(f_n).\]

    (b''). For any $f\in l(A)$ such that $f|_{A_{\infty,m}}\in C(A_{\infty,m})$, $\forall m\geq 1$, there exists a sequence $\{f_n\}_{n\geq 1}$ such that $f_n\in l(A_n)$, $f_n|_{A_{n,m}}\in C(A_{n,m}),\forall n,m\geq 1$, $f_n\stackrel{,}{\rightarrowtail} f$ and 
    \[\mcE(f)=\lim_{n\to\infty} \mcE_n(f_n).\] 
\end{theorem}

The proof of Theorem \ref{thm71} is exactly the same as that of Theorem \ref{thm213}. We omit it here. 

In particular, by applying the result for $\lambda\in (0,\frac{1}{2}),s\in [1,\infty)$ cases, we can immediate get the existence of a self-similar resistance form on  $\bigcup_{w\in W_*} F_{w,\lambda}\partial\blacktriangle$ by a same proof as in Section 6.1.

To make the resistance forms on $\bigcup_{w\in W_*} F_{w,\lambda}\partial\blacktriangle$ a strongly local regular Dirichlet form on $\AG_\lambda$, one can follow a standard argument developed by Kumagai \cite{kum2}, see also Section 3.4 of \cite{ki3}. One preparation we need here is to show that piecewise harmonic functions, which means the minimal energy extension of $f\in l(V_{m,\lambda})$ for some $m\geq 1$, extend to be continuous functions on $\AG_\lambda$. In fact, it is easy to see that for each point in $x\in\AG_{\lambda}$, $\{|f(z)-f(z')|:z,z'\in B_\delta(x)\cap \bigcup_{w\in W_*}F_w\partial\blacktriangle\}\to 0$ as $\delta\to 0$. We also need additonal assumptions on the measure $\mu_\lambda$, which should be similar to that in Section 3.4 of \cite{ki3}, noticing that we have the estimates Lemma \ref{lemma27} and Proposition \ref{prop54} to get a lower bound estimate of the energy (added at each level).  

Finally, the uniqueness depends only on Proposition \ref{prop54}, where the condition $s<1$ was not involved. 

\section*{Acknowledgments}
The author is grateful to Professor Robert S. Strichartz for his continued support and encouragement for me to work on this problem.

\bibliographystyle{amsplain}

\end{document}